\documentclass[a4paper]{article}
\usepackage{latexsym,bm}
\usepackage{amssymb,amsmath,amsthm}
\usepackage[english,german]{babel}
\usepackage{cite}
\usepackage[colorlinks=true]{hyperref}
\usepackage{color}

\newtheorem{theorem}{Theorem}[section]
\newtheorem{lemma}[theorem]{Lemma}
\newtheorem{remark}[theorem]{Remark}
\newtheorem{proposition}[theorem]{Propositon}
\newtheorem{definition}[theorem]{Definition}
\newtheorem{corollary}[theorem]{Corollary}
\numberwithin{equation}{section}
\hypersetup{linkcolor=blue,urlcolor=red,citecolor=red}
\allowdisplaybreaks[4] \numberwithin{equation}{section}

\title{Switching properties of  time optimal controls for systems of
 heat equations coupled by constant matrices\thanks{This work was partially supported by the National Natural Science Foundation of China  under grants 11601377, 11901432, 11971022.}}
\author{Shulin Qin\thanks{School of Science, Tianjin University of Commerce, Tianjin 300134, China (\texttt{shulinqin@yeah.net})}\and Gengsheng Wang\thanks{Center for Applied Mathematics, Tianjin University, Tianjin 300072, China (\texttt{wanggs62@yeah.net})}\and Huaiqiang Yu\thanks{School of Mathematics, Tianjin University, Tianjin 300354, China (\texttt{huaiqiangyu@tju.edu.cn})}}

\date{}
\begin{document}
\selectlanguage{english}
\maketitle
\begin{abstract}
This paper studies the time optimal control problem for systems of heat equations coupled
by a pair of  constant matrices. The control constraint is of the ball-type, while the target is the origin of the state space.
We obtain an upper bound for the number of switching points
of the optimal control over each interval with a fixed length. Also, we prove that at each  switching point,
the optimal control jump from one direction to the reverse direction.

\end{abstract}

\section{Introduction}

We start with introducing notation:
    Write $\mathbb{R}^+:=(0,+\infty)$, $\mathbb{N}:=\{0,1,\ldots\}$ and $\mathbb{N}^+:=\{1,2,\ldots\}$.
    Let
    $\Omega\subset\mathbb{R}^N$ (with $N\in\mathbb{N}^+$) be a bounded domain with a $C^2$ boundary $\partial\Omega$. Let
    $\omega\subset\Omega$ be a nonempty and open subset, with its characteristic function
    $\chi_\omega$.
     Let $\triangle$ be the Laplace operator
    with its domain $D(\triangle):=H_0^1(\Omega)\cap H^2(\Omega)$.
      Let
    $\mathbb{I}_k$ (with $k\in \mathbb{N}^+$) be the $k\times k$ identity matrix.
    Given a square matrix $D$, we write $\sigma(D)$ for its spectral.
    Denote by $B_1^m(0)$ (with $m\in \mathbb{N}$) the closed unit ball in
    $L^2(\Omega;\mathbb R^m)$ centered at the origin. Given two Banach spaces $K$ and $F$, write $\mathcal{L}(K;F)$  for the space of all linear bounded  operators from  $K$ to $F$.
    Given $a\in \mathbb{R}$, let $[a]$ be the largest integer less than or equal to  $a$.
    Given a subset $E\subset\mathbb{R}^+$, denote respectively by  $\sharp[E]$ and $|E|$ the cardinality and the measure (if it is measurable) of $E$.
    Write $L^\infty_{x,t}$ for the space $L^\infty(\Omega\times \mathbb{R}^+)$.
     Given $T>0$, let
\begin{eqnarray}\label{def-PC-func}
    \mathcal{PC}([0,T);L^2(\Omega;\mathbb R^m)) &:=& \bigg\{f:[0,T)\to L^2(\Omega;\mathbb R^m)
    ~\big|~ f \;\;\mbox{ has  at most finite discontinuities}
     \nonumber\\
     & &~~~~\mbox{which are of the first kind} \bigg\}.
   \end{eqnarray}
(Here, a discontinuity $\hat t$ of $f$ is called to be of the first kind, if both $\lim_{t\rightarrow \hat t^-} f(t)$
and $\lim_{t\rightarrow \hat t^+} f(t)$ exist.)

\subsection{Control problem}\label{yu-subsection-1-1}
    Let   $A\in\mathbb{R}^{n\times n}$ and $B\in\mathbb{R}^{n\times m}\setminus\{0\}$
    (with $n,m\in \mathbb{N}^+$). Consider the system of coupled heat equations:
\begin{equation}\label{yu-5-16-1}
\begin{cases}
   y_t=(\mathbb{I}_n\triangle+A)y+\chi_{\omega}Bu &\mbox{in}\;\;\Omega\times \mathbb{R}^+,\\
   y=0&\mbox{on}\;\;\partial\Omega\times \mathbb{R}^+,\\
   y(0)=y_0\in L^2(\Omega;\mathbb R^n),
\end{cases}
\end{equation}
   where $u\in L^\infty(\mathbb{R}^+;L^2(\Omega;\mathbb R^m))$. We will treat the solution to
   \eqref{yu-5-16-1} as a function from $[0,+\infty)$ to $L^2(\Omega;\mathbb R^m)$ and denoted it by $y(\cdot;y_0,u)$. Let
    \begin{eqnarray}\label{abbrev-operators}
     \mathcal A := \mathbb{I}_n\triangle+A
     \;\;\mbox{and}\;\;
     \mathcal B := \chi_\omega B.
    \end{eqnarray}
    One can directly check that the operator   $\mathcal A$, with its domain $ H^1_0(\Omega;\mathbb R^n)\cap H^2(\Omega;\mathbb R^n)$, generates an analytic
     semigroup $\{e^{t\mathcal A}\}_{t\geq 0}$ on $L^2(\Omega;\mathbb R^n)$. Then for each $y_0\in L^2(\Omega;\mathbb R^n)$ and each $u\in L^\infty(\mathbb{R}^+;L^2(\Omega;\mathbb R^m))$,
     \begin{eqnarray*}
      y(t;y_0,u)= e^{t\mathcal A} y_0
         + \int_0^t  e^{(t-s)\mathcal A} \mathcal B u(s)ds,
         ~~t\geq 0.
     \end{eqnarray*}

\par
    We next introduce our time optimal control problem:
\begin{equation}\label{original-tp}
    (\mathcal{TP})_{y_0}:\;\;\;\;T^*_{y_0}:=\inf \Big\{\hat{t}>0  ~:~
    \exists\;u\in L^\infty(\mathbb{R}^+;B^m_1(0))
    \;\;\mbox{s.t.}\;y(\hat{t};y_0,u)=0 \Big\},
\end{equation}
where $y_0\in L^2(\Omega;\mathbb R^n)\setminus\{0\}$ is the initial state,  $B_1^m(0)$ is the control constraint set, $\{0\}\in L^2(\Omega;\mathbb R^n)$ is the target set.
    {\it In the above problem, the number $T_{y_0}^*$ is called the optimal time;
     $u\in L^\infty(\mathbb{R}^+;B_1^m(0))$ is called an admissible control  if there is $t\in\mathbb{R}^+$ so that $y(t;y_0,u)=0$;    $u^*_{y_0}\in L^\infty(\mathbb{R}^+;B_1^m(0))$ is called an optimal
       control if $y(T^*_{y_0};y_0,u^*_{y_0})=0$ and $u^*_{y_0}(\cdot)=0$ over $(T^*_{y_0},+\infty)$.
        (The effective domain of $u^*_{y_0}$ is $[0,T^*_{y_0}]$.)
         Thus,  the optimal control to $(\mathcal{TP})_{y_0}$ is unique, if any two optimal controls coincide a.e. over $[0,T^*_{y_0}]$.}

\vskip 5pt
\par
   The main assumption of this paper is as: the initial state $y_0\in L^2(\Omega;\mathbb R^n)\setminus\{0\}$ satisfies
\vskip 5pt
    \noindent ~~~~$\textbf{Assumption (\mbox{A})}_{y_0}$: \emph{The problem $(\mathcal{TP})_{y_0}$ has an admissible control.}
\vskip 5pt
\noindent Several notes on the assumption $\textbf{(\mbox{A})}_{y_0}$ are given in order.
\begin{enumerate}
\item [($a_1$)] The reason that we ask $y_0\neq 0$ is as: when $y_0=0$, the problem $(\mathcal{TP})_{y_0}$
is trivial.
\item [($a_2$)] The assumption $\textbf{(\mbox{A})}_{y_0}$ is equivalent to
that the problem $(\mathcal{TP})_{y_0}$ has an optimal control. (See Theorem 3.11 in \cite[Chapter 3]{Wang-Wang-Xu-Zhang}.)
   \item [($a_3$)] In many cases, $y_0\in L^2(\Omega;\mathbb R^n)\setminus\{0\}$ satisfies $\textbf{(\mbox{A})}_{y_0}$. For instance,  according to \cite[Theorem 3.1]{Phung-Wang-Zhang},
   any $y_0\in L^2(\Omega;\mathbb R^n)\setminus\{0\}$ holds $\textbf{(\mbox{A})}_{y_0}$, provided that
     the system (\ref{yu-5-16-1})  is  null controllable on some  $[0,T]$ and
   $\|e^{t(\mathbb{I}_n\triangle+A)}\|_{\mathcal{L}(L^2(\Omega;\mathbb R^n);L^2(\Omega;\mathbb R^n))}\leq 1$ for each $t\in\mathbb{R}^+$. For more studies on this issue, we refer the readers
    to \cite[Chapter 3]{Wang-Wang-Xu-Zhang}.

\end{enumerate}

We end this subsection with introducing the following subspace:
\begin{equation}\label{yu-5-16-5}
    \mathfrak{L}:=L^2(\Omega;\mathfrak R)
    \;\;\mbox{where}\;\;
    \mathfrak{R}:=  \Bigg\{ \sum_{j=0}^{n-1}A^jBv_j: \{v_j\}_{j=0}^{n-1}\subset\mathbb{R}^m
    \Bigg\}.
\end{equation}
The space $\mathfrak{L}$ is indeed the controllable subspace of the system (\ref{yu-5-16-1}).
 Our Corollary \ref{yu-lemma-5-17-2} says
that if  $y_0$ satisfies $\textbf{(\mbox{A})}_{y_0}$,  then $y_0\in\mathfrak{L}$.

\subsection{Main results}

    We start with the next definition.
\begin{definition}\label{yu-definition-5-25-2}
    Let $T>0$ and $u \in \mathcal{PC}([0,T);L^2(\Omega;\mathbb R^m))$. The number $\hat t\in(0,T)$ is said to be a switching point of $u$, if both $\lim_{t\to \hat{t}^-} u(t)$ and $\lim_{t\to \hat{t}^+} u(t)$ exist
    and
     $\displaystyle\lim_{t\to \hat{t}^-} u(t) \neq \displaystyle\lim_{t\to \hat{t}^+} u(t)$.
    \end{definition}
\par
We next  introduce two important numbers $d_A$ and $q_{A,B}$:
\begin{equation}\label{yu-5-23-12}
    d_A:=\min\left\{\pi/|\mbox{Im}\lambda|:\lambda\in \sigma(A)\right\};
    \end{equation}
\begin{equation}\label{yu-5-23-13}
    q_{A,B}:=\max \bigg\{ \mbox{rank}\, (b,Ab,\ldots,A^{n-1}b) ~:~ b\;\mbox{is a column of}\;B
    \bigg\}.
\end{equation}
In (\ref{yu-5-23-12}), we agree that $1/0=+\infty$, consequently, we have $d_A=+\infty$, when
     $\sigma(A)\subset\mathbb{R}$. In \eqref{yu-5-23-13}, we have  $q_{A,B}\leq n$. The numbers $d_A$ and $q_{A,B}$ were introduced in \cite{Qin-Wang}, where the controllability of impulse controlled
systems of heat equations coupled by constant matrices was studied.

\par
    The main results are now stated as follows:
\begin{theorem}\label{yu-theorem-5-24-1}
    Suppose  $y_0 \in L^2(\Omega;\mathbb R^n)\setminus\{0\}$ satisfies the assumption $(\textbf{A})_{y_0}$.
     Then the following conclusions are true:
\begin{enumerate}
  \item [(i)] The problem $(\mathcal{TP})_{y_0}$ has a unique optimal control $u^*_{y_0}$ satisfying $\|u^*_{y_0}(t)\|_{L^2(\Omega;\mathbb R^m)}=1$ for a.e. $t\in(0, T^*_{y_0})$ (i.e., it has the bang-bang property).
      Moreover, the restriction of  $u^*_{y_0}$ over $[0,T^*_{y_0})$ is in the space $\mathcal{PC}([0,T^*_{y_0});B_1^m(0))$, which is given by \eqref{def-PC-func};
  \item [(ii)] For any open interval $I\subset(0,T^*_{y_0})$ with $|I|\leq d_A$, the optimal control $u^*_{y_0}$ has at most $(q_{A,B}-1)$ switching points in $I$, where $d_A$ and $q_{A,B}$ are given by (\ref{yu-5-23-12}) and (\ref{yu-5-23-13}), respectively.;

  \item [(iii)] Let $\hat t\in (0,T^*_{y_0})$ be a switching point of $u^*_{y_0}$.  Then
\begin{equation*}\label{yu-5-24-5}
    \lim_{t\to \hat{t}^-}u^*_{y_0}(t)+\lim_{t\to \hat{t}^+}u^*_{y_0}(t)=0.
\end{equation*}
\end{enumerate}
\end{theorem}

    Several notes on Theorem \ref{yu-theorem-5-24-1} are listed in order.
\begin{enumerate}
  \item [($b_1$)] From $(i)$ of Theorem \ref{yu-theorem-5-24-1}, we  see that
$u^*_{y_0}\in \mathcal{PC}([0,T^*_{y_0});B_1^m(0))$. It is natural to ask the behaviour of
 $u^*_{y_0}$ at $T^*_{y_0}$. Unfortunately, this is a very hard problem for us.

\item[($b_2$)]
In $(ii)$ of Theorem \ref{yu-theorem-5-24-1}, we only give an upper bound for the number of switching points of $u^*_{y_0}$ {\it in any open subinterval $I\subset(0,T^*_{y_0})$ with $|I|\leq d_A$. }
How to get a global upper bound over $(0,T^*_{y_0})$ is extremely hard for us. However, for the special case that
$\sigma(A)\subset\mathbb{R}$, we have $d_A=+\infty$, and thus $u^*_{y_0}$ has at most $(q_{A,B}-1)$ switching points over the whole the interval $(0,T^*_{y_0})$.

\item[($b_3$)] For the pure heat equation on $\Omega$, i.e., $A=0$ and $B=\mathbb{I}_1$, we see from
            $(ii)$ of Theorem \ref{yu-theorem-5-24-1} that the corresponding time optimal control $u^*_{y_0}$ has no any switching point, and thus  $u^*_{y_0}$ is continuous over $[0,T^*_{y_0})$. Indeed, in this case, $d_{A}=+\infty$ (see the statement in ($b_2$)) and $q_{A,B}=1$.
           {\it From this, we can say that the coupling causes switching points. }
\item [($b_4$)] The conclusion $(iii)$ in Theorem \ref{yu-theorem-5-24-1} says that
    the optimal control  jumps from one direction to its reverse direction at each switching point.

\end{enumerate}

\subsection {Comparison with related works}
To our best knowledge, the studies on the switching points for time optimal controls governed by
PDEs have not been touched upon. There have been some literatures on the related studies
for ODEs, for instance, \cite{Pontryagin, Poggiolini, Qin-Wang-Yu, Sontag, Sussmann} and references therein.
We would like mention, in particular, the work \cite{Qin-Wang-Yu}, where the similar problem was studied
and the similar results were  obtained for ODEs. However, it is not   easy   to extend
results from finite-dimensional systems to infinite-dimensional systems. Indeed, to study the  switching points
for the problem $(\mathcal{TP})_{y_0}$, we built up an $L^\infty$ null controllability for \eqref{yu-subsection-1-1}, used some point-wise unique continuation  to the dual system of (\ref{yu-5-16-1}),
and utilized some results  obtained in   \cite{Qin-Wang}.
 With regard to time optimal controls for parabolic equations, we would like mention \cite{Barbu, Carja, Cannarsa-Carja, Gozzi-Loreti, Wang-2008, Lv-Wang, Wang-Wang,  Wang-Zheng, Wang-Zuazua, Yu} and the references therein.

\subsection{Plan of this paper}
    The rest of the paper is organized as follows: Section 2
    gives some auxiliary results; Section 3 proves the main theorem; Section 4 presents an example.

\section{Auxiliary results}

\subsection{Decomposition of the system}

This subsection  presents a decomposition of the system \eqref{yu-5-16-1}  from perspective of the controllability. We starts with introducing the following  well-known Kalman controllability decomposition
for ODEs (see  \cite[Lemma 3.3, p.93]{Sontag}):
 \begin{lemma}
 \label{lem-Kalman-controllability-decomposition}
 Let $\mathfrak R$ be given in  (\ref{yu-5-16-5}) with $k\triangleq$dim\,$\mathfrak R$. Then there is an invertible matrix $P\in \mathbb R^{n\times n}$ with $P^\top=P^{-1}$ and four matrices $A_1\in\mathbb{R}^{k\times k}$, $A_2\in\mathbb{R}^{k\times (n-k)}$, $A_3\in\mathbb{R}^{(n-k)\times (n-k)}$,  $B_1\in\mathbb{R}^{k\times m}$ so that
\begin{equation}\label{yu-5-18-2}
P^{-1} \mathfrak R=\mathbb R^k \times\{0\},~~
    P^{-1}AP =\left(
                \begin{array}{cc}
                  A_1 & A_2 \\
                  0 & A_3 \\
                \end{array}
              \right)
\;\;\mbox{and}\;\;
              P^{-1}B=\left(
                               \begin{array}{c}
                                 B_1 \\
                                 0 \\
                               \end{array}
                             \right),
\end{equation}
and so that
\begin{equation}\label{yu-5-18-3}
    \mbox{rank}\,(B_1,A_1B_1,\cdots,A_1^{k-1}B_1)=k.
\end{equation}
(Here, it is agreed that $A_2,A_3$ are not there if $k=n$.)
 \end{lemma}

With the help of  Lemma \ref{lem-Kalman-controllability-decomposition}, we have the following decomposition for  the system (\ref{yu-5-16-1}):

\begin{proposition}\label{prop-relation-semigroups}
Let the matrices $P$, $\{A_j\}_{j=1}^3$ and $B_1$ be given in Lemma \ref{lem-Kalman-controllability-decomposition}. Then for each $t\geq0$,
\begin{eqnarray}\label{relation-C0-two}
 P^{-1}e^{t\mathcal A}P = \left(
                \begin{array}{cc}
                  e^{t(\mathbb{I}_k\triangle+A_1 )} & M(t)\\
                  0 & e^{t(\mathbb{I}_{n-k}\triangle+A_3)} \\
                \end{array}
              \right)
 \;\;\mbox{and}\;\;
~P^{-1} e^{t\mathcal A} \mathcal B=
\left(
 \begin{array}{c}
  e^{t(\mathbb{I}_k\triangle+A_1 )}\chi_\omega B_1 \\
  0  \\
 \end{array}
\right),
\end{eqnarray}
where the operator $M(t)$ is as:
\begin{eqnarray*}
M(t):=  e^{t \mathbb I_k \triangle } \int_0^t e^{(t-s)A_1} A_2  e^{s A_3} ds.
\end{eqnarray*}

\end{proposition}

\begin{proof}
Arbitrarily fix $t\geq 0$.
Since $e^{t\mathcal A}=e^{t\mathbb I_n \triangle} e^{tA}$ (see for instance \cite[Proposition 3.1]{Qin-Wang}), it follows from the first equality in (\ref{yu-5-18-2}) that
\begin{eqnarray*}
  P^{-1}e^{t\mathcal A}P &=& P^{-1} \left( e^{t\mathbb I_n \triangle}  e^{tA} \right)P
  =e^{t\mathbb I_n \triangle}  \left( P^{-1} e^{tA}P \right)
  \nonumber\\
  &=& \left(
        \begin{array}{cc}
          e^{t\mathbb I_k \triangle} & 0 \\
          0 & e^{t\mathbb I_{n-k} \triangle} \\
        \end{array}
      \right)
      \left(
                \begin{array}{cc}
                  e^{tA_1 } & \int_0^t e^{(t-s)A_1} A_2  e^{s A_3} ds\\
                  0 & e^{tA_3} \\
                \end{array}
              \right)
  \nonumber\\
  &=& \left(
                \begin{array}{cc}
                  e^{t\mathbb{I}_k\triangle} e^{tA_1 } & e^{t\mathbb{I}_k\triangle} \int_0^t e^{(t-s)A_1} A_2  e^{s A_3} ds\\
                  0 & e^{t\mathbb{I}_{n-k}\triangle} e^{tA_3} \\
                \end{array}
              \right)
  =\left(
                \begin{array}{cc}
                  e^{t(\mathbb{I}_k\triangle+A_1 )} & M(t)\\
                  0 & e^{t(\mathbb{I}_{n-k}\triangle+A_3)} \\
                \end{array}
              \right).
\end{eqnarray*}
This leads to the first equality in (\ref{relation-C0-two}). The second one in (\ref{relation-C0-two}) can be proved in a very similar way. This finishes the proof of Proposition \ref{prop-relation-semigroups}.
\end{proof}

\begin{corollary}\label{yu-lemma-5-17-2}
Suppose that  $y_0\in L^2(\Omega;\mathbb R^n)\setminus\{0\}$ satisfies
the assumption $(\textbf{A})_{y_0}$.
Then the following conclusions are true:
    \begin{enumerate}
      \item[(i)] It holds that $0<T^*_{y_0}<+\infty$;

      \item[(ii)] There is a unique $\hat y_0 \in L^2(\Omega;\mathbb R^k)$ so that  $P^{-1}y_0=(\hat y_0,0)^{\top}$, where $P$ is given in Lemma \ref{lem-Kalman-controllability-decomposition}.
          In particular,  $y_0\in \mathfrak{L}$, where  $\mathfrak{L}$ is given by \eqref{yu-5-16-5}.
    \end{enumerate}

\end{corollary}

\begin{proof}
From $(\textbf{A})_{y_0}$, one can easily derive the conclusion $(i)$ (see ($a_2$) in Section \ref{yu-subsection-1-1}). We next prove the conclusion $(ii)$.
    When $k=n$, we have $\mathfrak{L}=L^2(\Omega;\mathbb R^n)$, which leads to $(ii)$.
       We now suppose that $k<n$. Write
\begin{eqnarray}\label{y0-y1-y2}
 P^{-1} y_0 :=(y_1,y_2)^\top,
 \;\;\mbox{with}\;\;
 y_1 \in L^2(\Omega;\mathbb R^k)
 \;\;\mbox{and}\;\;
 y_2 \in L^2(\Omega;\mathbb R^{n-k}).
\end{eqnarray}
By $(\textbf{A})_{y_0}$, we can find $u\in L^\infty(\mathbb R^+;L^2(\Omega;\mathbb R^m))$ and $\hat t>0$ so that
\begin{eqnarray*}
0&=& y(\hat t; y_0,u)= e^{\hat t\mathcal A}y_0
+ \int_0^{\hat t}  e^{(\hat t-s)\mathcal A} \mathcal B u(s) ds
\nonumber\\
&=& P \Big[ \big( P^{-1} e^{\hat t\mathcal A} P \big) P^{-1}y_0
+ \int_0^{\hat t}  \big( P^{-1}e^{(\hat t-s)\mathcal A} \mathcal B \big) u(s) ds
\Big]
= P \left(
      \begin{array}{c}
        \cdots \\
        e^{\hat t(\mathbb{I}_{n-k}\Delta + A_3)} y_2  \\
      \end{array}
    \right).
\end{eqnarray*}
This yields that
\begin{eqnarray*}
0=e^{\hat t(\mathbb{I}_{n-k}\Delta + A_3)} y_2.
\end{eqnarray*}
   From this  and \cite[$(ii)$ of Proposition 3.2]{Qin-Wang}, it follows that $y_2=0$.
    Then the conclusion $(ii)$ follows from (\ref{y0-y1-y2}) and the first equation in (\ref{yu-5-18-2}).  This
    completes  the proof of Corollary \ref{yu-lemma-5-17-2}.
\end{proof}

\subsection{Null controllability of the system}
This subsection studies the $L^\infty_{x,t}$ null controllability for the system (\ref{yu-5-16-1}).
\begin{proposition}\label{yu-lemma-5-18-2}
   Let $\mathcal A$ and $\mathcal B$ be given by (\ref{abbrev-operators}).  Then the following three statements are equivalent:
\begin{enumerate}
  \item [(i)] It holds that $\mbox{rank}\;(B,AB,\cdots,A^{n-1}B)=n$;
  \item [(ii)]  The system (\ref{yu-5-16-1}),  with $L^\infty_{x,t}$-controls,  is null controllable, i.e., for each $T>0$ and each $y_0\in L^2(\Omega;\mathbb R^n)$, there is $u\in L^\infty(\Omega\times \mathbb R^+;\mathbb R^m)$ so that $y(T;y_0,u)=0$.
  \item [(iii)] There is $C=C(\Omega,\omega,A,B)>0$ so that for each $T>0$,
  \begin{eqnarray*}
    \| e^{T\mathcal A^*} z\|_{L^2(\Omega;\mathbb R^n)}
    \leq   Ce^{ \frac{C}{T} }
     \int_0^T \| \mathcal B^* e^{t \mathcal A^* } z \|_{ L^1(\Omega;\mathbb R^m) } dt\;\;\mbox{for each}\;\;
    z\in L^2(\Omega;\mathbb R^n).
  \end{eqnarray*}

\end{enumerate}
\end{proposition}

\begin{remark} The following two remarks are worth mentioning.
\begin{enumerate}
  \item [(i)]  The $L^2$-controllability of  coupled parabolic equations was obtained in \cite{Lissy-Zuazua}.
    Proposition \ref{yu-lemma-5-18-2} improve the corresponding result
  in \cite{Lissy-Zuazua}. We also mention the work
    \cite{Miller} for the $L^2$-controllability of an abstract heat-like equation.
  \item [(ii)] In this paper, we only need the null controllability over each $(0,T)$ for the system (\ref{yu-5-16-1})
  with  $L^\infty(0,T;L^2(\Omega;\mathbb{R}^m))$ controls. This is weaker than the controllability in Proposition
      \ref{yu-lemma-5-18-2}.
    However, the later   may have independent significance. This is the reason we present it here.

\end{enumerate}

\end{remark}

The following lemma is used to prove Proposition \ref{yu-lemma-5-18-2}.

\begin{lemma}\label{lem-new-LR-interpolation}
Suppose that $\mbox{rank}\;(B,AB,\cdots,A^{n-1}B)=n$. Let
\begin{eqnarray*}
 p := \min \bigg\{ j\in \mathbb N^+ ~:~
 \mbox{rank}\,(B,AB,\cdots,A^{j-1}B) =n
 \bigg\}.
\end{eqnarray*}
Then there is $C=C(\Omega,\omega,A,B)>0$ so that when $\theta\in(0,1)$ and $T\geq S> 0$,
  \begin{eqnarray}\label{new-LR-interpolation-ineq}
    \| e^{T\mathcal A^*} z\|_{L^2(\Omega;\mathbb R^n)}
    \leq   C e^{ \frac{C}{\theta T} }  \frac{1}{ S^{p-1} }
     \Bigg( \frac{1}{S} \int_{T-S}^T \|\mathcal B^* e^{t \mathcal A^* } z \|_{ L^1(\Omega;\mathbb R^m) } dt \Bigg)^{1-\theta}
     \Big(\|z\|_{L^2(\Omega;\mathbb R^n)} \Big)^{\theta}\;\mbox{for all}\;z\in L^2(\Omega;\mathbb R^n).
  \end{eqnarray}

\end{lemma}

\begin{proof}

Write $\{\lambda_j\}_{j\geq 1}$, with $\lambda_1<\lambda_2\leq\cdots$, for all eigenvalues
of  the operator $(-\triangle, H_0^1(\Omega)\cap H^2(\Omega))$. Let $ e_j$ (with $j=1,2,\dots$) be the  corresponding normalized eigenfunction.
We organize the proof by two steps.

\vskip 5pt

\noindent {\it Step 1. We show that  there is $C=C(\Omega,\omega,A,B)>0$ so that  when $t\in (0,1)$ and $\lambda>0$,
\begin{eqnarray}\label{coupled-LR-ineq}
 \| e^{t\mathcal A^*} z\|_{L^2(\Omega;\mathbb R^n)}
    \leq   C e^{ C \sqrt{\lambda} }  \frac{1}{ t^{p} }
      \int_0^t \|\mathcal B^* e^{s \mathcal A^* } z \|_{ L^1(\Omega;\mathbb R^m) } ds
\end{eqnarray}
for all $z=\sum_{\lambda_j\leq \lambda} z_j e_j$, with $\{ z_j \}_{j\in\{i\in\mathbb{N}^+:\lambda_i\leq \lambda\}} \subset \mathbb R^n$. }

For this purpose, we arbitrarily fix  $t\in (0,1)$, $\lambda>0$ and $z=\sum_{\lambda_j\leq \lambda} z_j e_j$ with $\{ z_j \}_{j\in\{i\in\mathbb{N}^+:\lambda_i\leq \lambda\}} \subset \mathbb R^n$.
 First, one can easily see that  for each $s\in(0,t)$,
\begin{eqnarray}\label{coupled-LR-ineq-1}
   e^{s\mathcal A^*} z= \sum_{\lambda_j\leq \lambda}
  (e^{-\lambda_j s}  e^{s A^\top} z_j)  e_j
  \;\;\mbox{and}\;\;
  B^\top e^{s\mathcal A^*} z= \sum_{\lambda_j\leq \lambda}
  ( e^{-\lambda_j s}  B^\top e^{s A^\top} z_j ) e_j.
\end{eqnarray}
Second, since $\mbox{rank}\;(B,AB,\cdots,A^{n-1}B)=n$, we can apply  \cite[Lemma 2]{Gyurkovics} (or \cite[Theorem 1]{Seidman-Yong}) to find  $C=C(A,B)>0$ (independent of $t\in(0,1)$) so that
\begin{eqnarray*}
 \| e^{t A^\top} v \|_{\mathbb R^n}
 \leq  \frac{C}{t^p} \int_0^t \| B^\top e^{s A^\top} v \|_{\mathbb R^m} ds,\;\mbox{when}\; v\in \mathbb R^n.
\end{eqnarray*}
From this and the first equation (\ref{coupled-LR-ineq-1}), it follows that
\begin{eqnarray}\label{coupled-LR-ineq-2}
 & &\| e^{t\mathcal A^*} z  \|_{L^2(\Omega;\mathbb R^n)}
 =
 \sum_{\lambda_j\leq\lambda}
 \| e^{-\lambda_j t}  e^{t A^\top} z_j \|_{\mathbb R^n}
 \leq   \sum_{\lambda_j\leq\lambda} e^{-\lambda_j t}   \Big( \frac{C}{ t^p}
 \int_0^t \| B^\top e^{s  A^\top} z_j  \|_{\mathbb R^m} ds
 \Big)
 \nonumber\\
 &\leq&  \frac{C}{ t^p} \int_0^t
  \sum_{\lambda_j\leq\lambda} \Bigg(
      \|  B^\top e^{-\lambda_j s} e^{s  A^\top} z_j  \|_{\mathbb R^m}
   \Bigg) ds
   =  \frac{C}{ t^p}   \int_0^t \| B^\top e^{s \mathcal A^*} z\|_{ L^2(\Omega;\mathbb R^m) } ds.
\end{eqnarray}

Meanwhile,  accroding to \cite[Theorems 5,3,8]{AEWZ} (see \cite{Lebeau-Robbiano} for the original study), there is $C=C(\Omega,\omega)>0$ so that for each sequence $\{a_j\}_{\lambda_j\leq\lambda} \subset \mathbb R$,
\begin{eqnarray*}
 \sum_{\lambda_j \leq \lambda} a_j^2  \leq  Ce^{C \sqrt{\lambda} }
 \bigg\| \sum_{\lambda_j \leq \lambda} a_j e_j \bigg\|_{ L^1(\omega) }^2.
\end{eqnarray*}
Set $v_j(s):= e^{-\lambda_j s}B^\top e^{s A^\top} z_j$ when $\lambda_j \leq \lambda$. The above (adapted to the vector valued case), along with the second equation in (\ref{coupled-LR-ineq-1}), yields that for each $s\in (0,t)$,
\begin{eqnarray*}
 \| B^\top  e^{s\mathcal A^*} z  \|_{L^2(\Omega;\mathbb R^m)}^2  &=& \bigg\| \sum_{\lambda_j\leq \lambda} v_j(s) e_j
 \bigg\|_{L^2(\Omega;\mathbb R^m)}^2=\sum_{\lambda_j\leq \lambda} \|v_j(s)\|_{\mathbb{R}^m}^2
 \nonumber\\
 &\leq&  Ce^{C \sqrt{\lambda} }
 \bigg\| \sum_{\lambda_j \leq \lambda} v_j(s) e_j \bigg\|_{ L^1(\omega;\mathbb R^m) }^2
 = Ce^{C \sqrt{\lambda} } \| \mathcal B^* e^{s\mathcal A^*} z  \|_{L^1(\omega;\mathbb R^m)}^2,
\end{eqnarray*}
which, together with (\ref{coupled-LR-ineq-2}), leads to (\ref{coupled-LR-ineq}).
\vskip 5pt

\noindent {\it Step 2. We use (\ref{coupled-LR-ineq}) to prove (\ref{new-LR-interpolation-ineq}).}

We only need to show (\ref{new-LR-interpolation-ineq}) for $z\neq 0$.
Arbitrarily fix $\lambda>0$, $\theta\in (0,1)$,  $0<S\leq T$  and $z\in L^2(\Omega;\mathbb R^n)\setminus\{0\}$.   Write
\begin{eqnarray*}
 z= \sum_{j\geq 1} z_j e_j
 = \sum_{\lambda_j\leq \lambda}  z_j e_j +
 \sum_{\lambda_j> \lambda}  z_j e_j
 := z_{\leq \lambda} + z_{>\lambda},
\end{eqnarray*}
   where $\{z_j\}_{j\in\mathbb{N}^+}\subset\mathbb{R}^n$. Set $S_1:=\min\{S,T/2\}$. It is clear that $S/2\leq S_1 \leq S$. By (\ref{coupled-LR-ineq}), where  $(t,z)$ is replaced by $({S_1, e^{(T-S_1)\mathcal{A}^*}}z)$,
   after some simple computations, we can find  $C_1$ and $C_2$
(only depending on $\Omega$, $\omega$, $A$ and $B)$) so that
\begin{eqnarray}\label{coupled-LR-ineq-3}
& &\| e^{T\mathcal A^*} z  \|_{L^2(\Omega;\mathbb R^n)}
\leq \| e^{T\mathcal A^*} z_{\leq\lambda }  \|_{L^2(\Omega;\mathbb R^n)}
+ \| e^{T\mathcal A^*} z_{>\lambda }  \|_{L^2(\Omega;\mathbb R^n)}
\nonumber\\
&\leq&  Ce^{C \sqrt{\lambda} }  \frac{1}{S_1^{p-1}}\frac{1}{S_1} \int_{T-S_1}^{T} \| \mathcal B^* e^{s\mathcal A^*} z_{\leq\lambda }  \|_{L^1(\Omega;\mathbb R^m)} ds
+ C_1e^{-\lambda T} \| z_{>\lambda} \|_{L^2(\Omega;\mathbb R^n)}
\nonumber\\
&\leq& \frac{C_2}{S^{p-1}} \Bigg[  e^{C \sqrt{\lambda} }  \left(\frac{1}{S} \int_{T-S}^{T}
\| \mathcal B^* e^{s\mathcal A^*} z  \|_{L^1(\Omega;\mathbb R^m)} ds\right)
+  e^{C\sqrt{\lambda}-\lambda T/2}  \|z\|_{L^2(\Omega;\mathbb R^n)}
 \Bigg].
\end{eqnarray}
Set $\beta := \frac{\theta}{1-\theta}$. Observe that when $\lambda>0$,
  \begin{eqnarray*}
  C \sqrt{\lambda} \leq  \frac{1}{4} \beta\lambda T  + \frac{C^2}{\beta T}
  \;\;\mbox{and}\;\;
  C \sqrt{\lambda} - \frac{1}{4} \lambda T  \leq \frac{C^2}{T}.
  \end{eqnarray*}
 These, along with (\ref{coupled-LR-ineq-3}), yield that
 \begin{eqnarray*}
\| e^{T\mathcal A^*} z  \|_{L^2(\Omega;\mathbb R^n)}
\leq \frac{C_2}{S^{p-1}} e^{ \frac{C^2}{\theta T} } \Bigg[  e^{\frac{1}{4} \beta\lambda T  }  \left(\frac{1}{S} \int_{T-S}^{T}
\| \mathcal B^* e^{s\mathcal A^*} z  \|_{L^1(\Omega;\mathbb R^m)} ds\right)
+  e^{- \frac{1}{4} \lambda T}  \|z\|_{L^2(\Omega;\mathbb R^n)}
 \Bigg].
\end{eqnarray*}
Since the above holds for all $\lambda>0$, we can see  that for each $\varepsilon\in(0,1)$,
 \begin{eqnarray}\label{yu-6-23-1}
\| e^{T\mathcal A^*} z  \|_{L^2(\Omega;\mathbb R^n)}
\leq \frac{C_2}{S^{p-1}} e^{ \frac{C^2}{\theta T} } \left[  \frac{1}{\varepsilon^\beta}  \left(\frac{1}{S} \int_{T-S}^{T}
\| \mathcal B^* e^{s\mathcal A^*} z  \|_{L^1(\Omega;\mathbb R^m)} ds\right)
+  \varepsilon \|z\|_{L^2(\Omega;\mathbb R^n)}
 \right].
\end{eqnarray}
   Meanwhile, we have
   \begin{eqnarray}\label{3.7wang}
   \int_{T-S}^{T}\|\mathcal B^* e^{s\mathcal A^*} z  \|_{L^1(\Omega;\mathbb R^m)} ds\neq 0.
   \end{eqnarray}
   Indeed, if (\ref{3.7wang}) were not true, then by  (\ref{yu-6-23-1}), we would have
$$
    \| e^{T\mathcal A^*} z  \|_{L^2(\Omega;\mathbb R^n)}
\leq \varepsilon\frac{C_2}{S^{p-1}} e^{ \frac{C^2}{\theta T} } \|z\|_{L^2(\Omega;\mathbb{R}^n)}\;\;\mbox{for all}\;\;\varepsilon\in(0,1).
$$
    Letting $\varepsilon\to 0^+$ in the above leads to  $\| e^{T\mathcal A^*} z  \|_{L^2(\Omega;\mathbb R^n)}=0$. This, together with \cite[$(ii)$ of Proposition 3.2]{Qin-Wang}, implies that $z=0$ which leads to a contradiction. So \eqref{3.7wang} is true.
\par
    In the case that $\|z\|_{L^2(\Omega;\mathbb R^n)}\leq \frac{1}{S} \int_{T-S}^{T}
\| \mathcal B^* e^{s\mathcal A^*} z  \|_{L^1(\Omega;\mathbb R^m)} ds$, one can directly check that
\begin{eqnarray*}
    \|e^{T\mathcal{A}^*}z\|_{L^2(\Omega;\mathbb{R}^n)} &\leq& C_3\|z\|_{L^2(\Omega;\mathbb{R}^n)}
    \leq C_3\left(\frac{1}{S} \int_{T-S}^{T}
\| \mathcal B^* e^{s\mathcal A^*} z  \|_{L^1(\Omega;\mathbb R^m)} ds\right)^{1-\theta}(\|z\|_{L^2(\Omega;\mathbb{R}^n)})^{\theta}\nonumber\\
 &\leq&\frac{C_4}{S^{p-1}}e^{\frac{C^2}{\theta T}}\left(\frac{1}{S} \int_{T-S}^{T}
\| \mathcal B^* e^{s\mathcal A^*} z  \|_{L^1(\Omega;\mathbb R^m)} ds\right)^{1-\theta}(\|z\|_{L^2(\Omega;\mathbb{R}^n)})^{\theta},
\end{eqnarray*}
   form some constants  $C_3$ and $C_4$ depending only on $A$. (Here we used  $T<1$.) Thus, (\ref{new-LR-interpolation-ineq}) is true in this case.
\par
    In the case that $\|z\|_{L^2(\Omega;\mathbb R^n)}> \frac{1}{S} \int_{T-S}^{T}
\| \mathcal B^* e^{s\mathcal A^*} z  \|_{L^1(\Omega;\mathbb R^m)} ds$, we let
\begin{equation*}\label{yu-6-23-2}
    \varepsilon_0:=\left(\frac{1}{S}\int_{T-S}^T\|\mathcal{B}^*e^{s\mathcal{A}^*}z\|_{L^1(\Omega;\mathbb{R}^m)}ds/
    \|z\|_{L^2(\Omega;\mathbb{R}^n)}\right)^{\frac{1}{1+\beta}}.
\end{equation*}
    By (\ref{3.7wang}), we have $\varepsilon_0\in (0,1)$.
    Thus we can use (\ref{yu-6-23-1}) where $\varepsilon=\varepsilon_0$ to find
\begin{eqnarray*}
     \|e^{T\mathcal{A}^*}z\|_{L^2(\Omega;\mathbb{R}^n)}
     \leq 2\frac{C_2}{S^{p-1}}e^{\frac{C^2}{\theta T}}\left(\frac{1}{S} \int_{T-S}^{T}
\| \mathcal B^* e^{s\mathcal A^*} z  \|_{L^1(\Omega;\mathbb R^m)} ds\right)^{\frac{1}{1+\beta}}(\|z\|_{L^2(\Omega;\mathbb{R}^n)})^{\frac{\beta}{1+\beta}}.
\end{eqnarray*}
    This, along with the fact $\beta=\frac{\theta}{1-\theta}$, leads to (\ref{new-LR-interpolation-ineq}) for this case.
\par
    In summary, we complete the proof of Lemma \ref{lem-new-LR-interpolation}.
\end{proof}

We now on the position to prove Proposition \ref{yu-lemma-5-18-2}.

\begin{proof}[Proof of Proposition \ref{yu-lemma-5-18-2}]
We divide the proof by several steps.
\vskip 5pt
\noindent{\it Step 1. We show that $(i)\Rightarrow(iii)$.}

     Suppose that $(i)$ is true. It suffices to prove $(iii)$ for the case that $0<T<1$.
     To this end, we arbitrarily fix $T\in (0,1)$ and $z\in L^2(\Omega;\mathbb R^n)$.
     By $(i)$, we can apply
  Lemma \ref{lem-new-LR-interpolation}, where
 $(T,S,\theta,z)$ is replaced by
$(T/2^{j+1},T/2^{j+1},1/3,e^{T\mathcal A^*/2^{j+1}}z)$, with $j$ a nonnegative integer, to find
  $C=C(\Omega,\omega,A,B)>0$ so that for each $\varepsilon>0$,
 \begin{eqnarray*}
 \| e^{T\mathcal A^*/2^j} z  \|_{L^2(\Omega;\mathbb R^n)}&\leq&
 C e^{ \frac{C2^j}{T} }\left(\int_0^{T/2^{j+1}}\|\mathcal{B}^*e^{(s+\frac{T}{2^{j+1}})\mathcal{A}^*}z\|_{L^1(\Omega;\mathbb{R}^m)}
 ds\right)^{2/3}\left(\| e^{T\mathcal A^*/2^{j+1} } z\|_{L^2(\Omega;\mathbb R^n)} \right)^{1/3}\nonumber\\
&\leq&    C e^{ \frac{C2^j}{T} }
     \left( \int_{T/2^{j+1}}^{T/2^j}
\| \mathcal B^* e^{s\mathcal A^*} z  \|_{L^1(\Omega;\mathbb R^m)} ds \right)^{2/3}
     \left(\| e^{T\mathcal A^*/2^{j+1} } z\|_{L^2(\Omega;\mathbb R^n)} \right)^{1/3}
    \nonumber\\
&\leq& C^{\frac{3}{2} } e^{ \frac{3C2^j}{2T} }  \frac{2}{3\sqrt{3\varepsilon}}  \int_{T/2^{j+1}}^{T/2^j}
\| \mathcal B^* e^{s\mathcal A^*} z  \|_{L^1(\Omega;\mathbb R^m)} ds
+  \varepsilon \|e^{T\mathcal A^*/2^{j+1} } z\|_{L^2(\Omega;\mathbb R^n)}.
\end{eqnarray*}
(In the above, we used the Young inequality.)
 Choosing $\varepsilon=e^{-\frac{3C2^j}{T} }$ in the above leads to  that for all $j\geq 0$,
\begin{eqnarray*}
 e^{-\frac{3C2^j}{T} } \| e^{T\mathcal A^*/2^j} z  \|_{L^2(\Omega;\mathbb R^n)}
 -  e^{-\frac{3C2^{j+1}}{T} }
 \| e^{T\mathcal A^*/2^{j+1} } z  \|_{L^2(\Omega;\mathbb R^n)}
 \leq  \frac{2}{3\sqrt{3}}C^{\frac{3}{2} }
 \int_{T/2^{j+1}}^{T/2^j}
\| \mathcal B^* e^{s\mathcal A^*} z  \|_{L^1(\Omega;\mathbb R^m)} ds.
\end{eqnarray*}
Summing the above from $j=0$ to $+\infty$  yields
\begin{eqnarray*}
  \| e^{T\mathcal A^*} z  \|_{L^2(\Omega;\mathbb R^n)}
 \leq  \frac{2}{3\sqrt{3}}C^{\frac{3}{2} } e^{\frac{3C}{T} }
 \int_{0}^{T}
\| \mathcal B^* e^{s\mathcal A^*} z  \|_{L^1(\Omega;\mathbb R^m)} ds,
\end{eqnarray*}
which leads to  $(iii)$.
\vskip 5pt
\noindent{\it Step 2. We show that $(iii)\Rightarrow(ii)$.}

This follows by  the classical duality method. We omit the details here.
\vskip 5pt
\noindent{\it Step 3. We show that $(ii)\Rightarrow(i)$.}

 By contradiction, we suppose that $(ii)$ holds but
\begin{eqnarray*}
 \mbox{rank}\,(B,AB,\ldots,A^{n-1}B)<n.
\end{eqnarray*}
Then there is
$\vartheta\in \mathbb R^n\setminus\{0\}$ so that
\begin{eqnarray}\label{0620-special-initial-data-1}
  \langle \vartheta, v \rangle_{\mathbb R^n}=0\;\;\mbox{for all}\;\; v\in \mathfrak R.
\end{eqnarray}
(Here, $\mathfrak R$ is given by (\ref{yu-5-16-5}).) Set
\begin{eqnarray}\label{0620-special-initial-data}
 y_0(x):=\vartheta~~\mbox{for a.e.}\;\;x\in\Omega.
\end{eqnarray}
According to $(ii)$, there is a control $u\in L^\infty(\Omega\times\mathbb R^+;\mathbb R^m)$ so that
$y(T;y_0,u)=0$.
Let $\varepsilon>0$ so that
$\varepsilon u \in L^\infty( \mathbb R^+;B_1^m(0))$.
Then $\varepsilon u$ is an admissible control to the
problem $(\mathcal{TP})_{\varepsilon y_0}$. From this and $(ii)$ of
 Corollary \ref{yu-lemma-5-17-2}, we see that
 $y_0\in \mathfrak{L}\setminus\{0\}$
 (where $\mathfrak{L}$
 is given by \eqref{yu-5-16-5}).
 This contradicts (\ref{0620-special-initial-data}) and (\ref{0620-special-initial-data-1}).
 Therefore, $(i)$ is true.

 Hence, we finish  the proof of Proposition \ref{yu-lemma-5-18-2}.
\end{proof}

\subsection{Unique continuation of the dual system}

This subsection presents a continuation property for the dual system of \eqref{yu-5-16-1}.
\begin{proposition}\label{thm-sampling-unique}
Let $I\subset \mathbb R^+$ be an open interval  with $|I|\leq d_A$,
where $d_A$  is given by (\ref{yu-5-23-12}).  Suppose that there is $z\in L^2(\Omega;\mathbb R^n)$ and  $\{t_j\}_{j=1}^{q_{A,B}} \subset I$
(where  $q_{A,B}$ is given by  (\ref{yu-5-23-13}))
so that
\begin{eqnarray}\label{samping-points-q}
  \mathcal B^* e^{t_j \mathcal A^*} z = 0
  \;\mbox{for each}\;\;j=1,\ldots,q_{A,B}.
\end{eqnarray}
Then
\begin{eqnarray}\label{yu-6-23-3}
 z(x) \in \mbox{Ker\,} \left(
                         \begin{array}{c}
                            B^\top\\
                           B^\top A^\top \\
                           \vdots \\
                           B^\top (A^\top)^{n-1} \\
                         \end{array}
                       \right)
 \;\;\mbox{for a.e.}\;\;
 x\in\Omega.
\end{eqnarray}
 Especially, $\mathcal B^*e^{t\mathcal A^*}z = 0$ for each $t\geq0$.

\end{proposition}
\begin{remark}\label{remark2.8wang}
Suppose that $(A,B)$ satisfies the kalman rank condition. Then the right hand side of \eqref{yu-6-23-3} is $\{0\}$.
Thus, by Proposition \ref{thm-sampling-unique}, we see that \eqref{samping-points-q} implies $z=0$ over $\Omega$.
This is a unique continuation property of the dual system of \eqref{yu-5-16-1}.
\end{remark}
Before proving Proposition \ref{thm-sampling-unique}, we  recall the following two lemmas:
\begin{lemma}\label{yu-lemma-5-23-2} (\cite[Theorem 5.2]{Qin-Wang})
    Let $\{t_j\}_{j=1}^p\subset(0,+\infty)$ with $p\in \mathbb{N}^+$.  Then the following two statements are equivalent:
\begin{enumerate}
  \item [(i)] It holds that $\mbox{rank}\;(e^{At_1}B,e^{At_2}B,\cdots,e^{At_p}B)=n$;
  \item [(ii)] If $z\in L^2(\Omega;\mathbb R^n)$  satisfies that
  \begin{eqnarray*}
    \mathcal B^* e^{t_j \mathcal A^*} z=0 \;\;\mbox{in}\;\;  L^2(\Omega;\mathbb R^m)\;\;\mbox{for all}\;\;j\in\{1,2,\ldots,p\},
  \end{eqnarray*}
  then $z=0$.
\end{enumerate}

\end{lemma}

\begin{lemma}(\cite[Theorem 2.2]{Qin-Wang})
\label{yu-lemma-5-23-4}
Let $d_A$ and $q_{A,B}$ be given by (\ref{yu-5-23-12}) and (\ref{yu-5-23-13}), respectively.
   Then for each increasing sequence $\{t_j\}_{j=1}^{q_{A,B}}\subset\mathbb{R}$ with $t_{q_{A,B}}-t_1<d_A$,
\begin{eqnarray*}\label{yu5-23-14}
    \mbox{rank}\;(e^{At_1}B,e^{At_2}B,\cdots,e^{At_{q_{A,B}}}B)=
\mbox{rank}\;(B,AB,\cdots,A^{n-1}B).
\end{eqnarray*}

\end{lemma}

\begin{remark}
 $(a)$ We mention that $(ii)$ of Lemma \ref{yu-lemma-5-23-2} can be replaced by an interpolation inequality.  See   \cite[Proposition 1.3]{Wang-Yan-Yu} for details;
  $(b)$ The number $d_A$ is necessary to ensure Lemma \ref{yu-lemma-5-23-4}. Besides, the optimality of the number $q_{A,B}$ is stressed in some sense in \cite[Theorem 2.2]{Qin-Wang}.
\end{remark}
We now on the position to prove Proposition \ref{thm-sampling-unique}.
\begin{proof}[Proof of Proposition \ref{thm-sampling-unique}]
Let the matrices $P$, $\{A_j\}_{j=1}^3$ and $B_1$ be given in Lemma \ref{lem-Kalman-controllability-decomposition}. Let $z_1\in L^2(\Omega;\mathbb R^k)$ and $z_2 \in L^2(\Omega; \mathbb R^{n-k})$ satisfy that
\begin{eqnarray}\label{samping-points-q-1}
 P^\top z=(z_1,z_2)^\top.
\end{eqnarray}
By the second equality in (\ref{relation-C0-two}), it follows that for each $t\geq0$,
\begin{eqnarray}\label{samping-points-q-2}
 \mathcal B^* e^{t\mathcal A^*}z &=&
 \mathcal B^* e^{t\mathcal A^*} (P^{-1})^\top (z_1,z_2)^\top
 = ( P^{-1} e^{t \mathcal A} \mathcal B )^* (z_1,z_2)^\top
 \nonumber\\
 &=&\left( \chi_\omega B_1^\top e^{(T-t)(\mathbb I_k\triangle+ A_1^\top)}, 0 \right)
  (z_1,z_2)^\top
 = \chi_\omega B_1^\top e^{t(\mathbb I_k\triangle+ A_1^\top)} z_1.
\end{eqnarray}
From this and (\ref{samping-points-q}), one has that
\begin{eqnarray*}\label{samping-points-q-3}
  \chi_\omega B_1^\top e^{t_j(\mathbb I_k\triangle+ A_1^\top)} z_1 =0,\;\;j=1,\ldots,q_{A,B}.
\end{eqnarray*}
Using (\ref{yu-5-18-3}) and noting that $d_A\leq d_{A_1}$ and $q_{A,B}\geq q_{A_1,B_1}$,  we can apply Lemmas \ref{yu-lemma-5-23-2}, \ref{yu-lemma-5-23-4}, where $(A,B)$ is replaced by $(A_1,B_1)$  to get
\begin{eqnarray}\label{samping-points-q-4}
 z_1=0.
\end{eqnarray}
This, along with (\ref{samping-points-q-1}) and Lemma \ref{lem-Kalman-controllability-decomposition}, yields that for a.e. $x\in\Omega$,
\begin{eqnarray*}
 \langle v,z(x) \rangle_{\mathbb R^n}=  \langle P^{-1}v,(z_1(x),z_2(x) )^\top \rangle_{\mathbb R^n}
 =0,\;\;\mbox{when}\;\; v\in \mathfrak R,
\end{eqnarray*}
(Here, $\mathfrak R$ is given by \eqref{yu-5-16-5}.)
which leads to (\ref{yu-6-23-3}).
\par
 Finally, from (\ref{samping-points-q-2}) and (\ref{samping-points-q-4}), it follows that
\begin{eqnarray*}
 \mathcal B^*e^{t\mathcal A^*}z = 0
 \;\;\mbox{for each}\;\;
 t\geq0.
\end{eqnarray*}
 This ends the proof of Proposition \ref{thm-sampling-unique}.
\end{proof}

The following result is a direct consequence of Proposition \ref{thm-sampling-unique}.
\begin{corollary}\label{cor-switch-behavior}
Let $T>0$ and $z\in L^2(\Omega;\mathbb R^n)$ satisfy
that $I_{z,T} \neq (0,T)$, where
\begin{eqnarray*}\label{dual-solu-nodel-set}
I_{z,T} := \big\{ t\in(0,T) ~:~ \mathcal B^*e^{(T-t) \mathcal A^*}z  = 0 \}.
\end{eqnarray*}
Then $I_{z,T}$ has at most $\big([T/d_A]+1 \big) (q_{A,B}-1)$ elements,
where $A$ and $q_{A,B}$ are given by (\ref{yu-5-23-12}) and (\ref{yu-5-23-13}), respectively.
\end{corollary}
\begin{proof}
    Since $I_{z,T} \neq (0,T)$, we have that $\mathcal{B}^*e^{(T-\cdot)\mathcal{A}^*}z$ is not a zero function over $\mathbb{R}^+$. By contradiction, we suppose that
    $$
    \sharp [I_{z,T}]> \big([T/d_A]+1 \big) (q_{A,B}-1).
    $$
     Then there is an open interval $\hat{I}\subset (0,T)$ with  $|\hat{I}|\leq d_A$ so that $\sharp[I_{z,T}\cap \hat{I}]\geq q_{A,B}$. Thus, by Proposition \ref{thm-sampling-unique}, we have that $\mathcal{B}^*e^{(T-\cdot)\mathcal{A}^*}\equiv 0$, which leads to
      a contradiction. This finishes the proof of Corollary \ref{cor-switch-behavior}.
\end{proof}

\subsection{Local maximum principle}

This section deals with the maximum principle of the problem $(\mathcal{TP})_{y_0}$.
It deserves mentioning that the standard maximum principle may not hold for $(\mathcal{TP})_{y_0}$ and what we get is the local maximum principle. (see  \cite[Chapter 4]{Wang-Wang-Xu-Zhang}.)

\begin{proposition}\label{thm-TP-PMP}
Let $y_0\in L^2(\Omega;\mathbb R^n)\setminus\{0\}$.
Then for each $T\in(0,T^*_{y_0})$, there is a  multiplier
$\xi_T\in \mathfrak{L}\setminus\{0\}$ (where $\mathfrak{L}$ is given by \eqref{yu-5-16-5}), with the property
\begin{equation}\label{adjoint-eq}
 \mathcal B^* e^{(T-\cdot)\mathcal A^*}\xi_T  \in L^1(0,T;L^2(\Omega;\mathbb R^m))\setminus\{0\},
\end{equation}
so that if $u^*_{y_0}$ is an optimal control of $(\mathcal{TP})_{y_0}$, then
\begin{eqnarray}\label{PMP-eq}
    \left\langle u^*_{y_0}(t),\mathcal B^* e^{(T-t)\mathcal A^*}\xi_T  \right\rangle_{L^2(\Omega;\mathbb R^m)}
=\max_{v\in B^m_1(0)}  \Big\langle v,\mathcal B^* e^{(T-t)\mathcal A^*}\xi_T \Big\rangle_{L^2(\Omega;\mathbb R^m)}
\;\;\mbox{for a.e.}\;\;
t\in (0,T).
\end{eqnarray}
\end{proposition}

\begin{remark}
 The multiplier $\xi_T$ in (\ref{PMP-eq}) is independent of the choice of the optimal controls to $(\mathcal{TP})_{y_0}$. The equality (\ref{PMP-eq}) is called the local maximum principle introduced in \cite[Section 4.2]{Wang-Wang-Xu-Zhang}.

\end{remark}
\par
Our strategy to prove Proposition \ref{thm-TP-PMP} is as follows:
We set up a new time optimal control problem which is equivalent to $(\mathcal{TP})_{y_0}$,
then build up the local  maximum principle for the new problem, and finally go back to
to  $(\mathcal{TP})_{y_0}$. To introduce the new problem, we
recall Lemma \ref{lem-Kalman-controllability-decomposition} for  $A_1$, $B_1$, $k$ and $P$.
Write $\hat{y}(\cdot;\hat{y}_0,u)$ for the solution to
 the following reduced control system:
 \begin{equation}\label{yu-5-18-9}
\begin{cases}
    \hat{y}_t=(\mathbb{I}_k\triangle+A_1)\hat{y}+B_1u&\mbox{in}\;\;
    \Omega\times\mathbb{R}^+,\\
    \hat{y}=0&\mbox{on}\;\;\partial\Omega\times\mathbb{R}^+,\\
    \hat{y}(0)=\hat{y}_0 \in L^2(\Omega;\mathbb R^k),
\end{cases}
\end{equation}
 where  $u$ is taken from $L^\infty(\mathbb R^+;L^2(\Omega;\mathbb R^m))$.
 The new time optimal control problem reads:
  \begin{equation}\label{yu-5-23-b-1}
    (\widehat{\mathcal{TP}})_{\hat{y}_0}:\;\;\;\;\widehat{T}^*_{\hat{y}_0}:=\inf\{\hat{t}>0
    ~:~  \exists\;u\in L^\infty(\mathbb{R}^+;B_1^m(0))
    \;\;\mbox{s.t.}\;\hat{y}(\hat{t};\hat{y}_0,u)=0\}.
\end{equation}
 {\it Notice that the new problem holds the state system \eqref{yu-5-18-9}, where
 $(A_1,B_1)$ satisfies the Kalman rank condition which plays an important role in getting the local maximum principle.}

The following lemma gives the equivalence between $(\mathcal{TP})_{y_0}$ and $(\widehat{\mathcal{TP}})_{\hat{y}_0}$.

\begin{lemma}\label{yu-proposition-5-18-1}
    Let $y_0\in L^2(\Omega;\mathbb R^n)\setminus\{0\}$ and  $\hat y_0\in L^2(\Omega;\mathbb R^k)\setminus\{0\}$
    satisfy $P^{-1}y_0=(\hat y_0,0)^{\top}$. Then
    the problems $(\mathcal{TP})_{y_0}$ and $(\widehat{\mathcal{TP}})_{\hat{y}_0}$ are equivalent, i.e.,
    they share the same optimal time and the same optimal controls (if one of them has an optimal control).
\end{lemma}

\begin{proof}
    We first claim that for each control $u\in L^\infty(\mathbb R^+;L^2(\Omega;\mathbb R^m))$,
    \begin{eqnarray}\label{relation-solutions}
     P^{-1}y(t;y_0,u)=(\hat y(t;\hat y_0,u),0)^{\top}
     \;\;\mbox{for each}\;\;t\geq0.
    \end{eqnarray}
To this end, we use Proposition \ref{prop-relation-semigroups} to see that when  $u\in L^\infty(\mathbb R^+;L^2(\Omega;\mathbb R^m))$,
\begin{eqnarray*}
 P^{-1} y(t;y_0,u) &=&
 P^{-1} e^{t\mathcal A}P P^{-1}y_0 + P^{-1}\int_0^t  e^{(t-s)\mathcal A} \mathcal B u(s) ds
 \nonumber\\
 &=& \left( P^{-1} e^{t\mathcal A}P \right) (\hat y_0,0)^{\top} + \int_0^t  \left( P^{-1}e^{(t-s)\mathcal A} \mathcal B \right) u(s) ds
 \nonumber\\
 &=& \left( e^{t(\mathbb{I}_k\Delta+A_1 )}\hat y_0,0 \right)^{\top}  + \int_0^t  \left( e^{(t-s)(\mathbb{I}_k\Delta+A_1 )}\chi_\omega B_1, 0 \right)^{\top} u(s) ds,
\end{eqnarray*}
which leads to (\ref{relation-solutions}).

 Next, from (\ref{relation-solutions}), we find that for each $u\in L^\infty(\mathbb R^+;L^2(\Omega;\mathbb R^m))$,
 \begin{eqnarray*}
  y(t;y_0,u)=0 \;\mbox{for each}\; t\geq 0\;\;\mbox{if and only if}\;\;
  \hat y(t;\hat y_0,u)=0 \;\mbox{for each}\; t\geq 0.
 \end{eqnarray*}
This, along with (\ref{original-tp}) and (\ref{yu-5-23-b-1}), shows that
      $(\mathcal{TP})_{y_0}$ and $(\widehat{\mathcal{TP}})_{\hat{y}_0}$
      are equivalent.

      Hence, we finish the proof of Lemma \ref{yu-proposition-5-18-1}.
      \end{proof}

The next lemma gives  the local maximum principle for the problem $(\widehat{\mathcal{TP}})_{\hat{y}_0}$.
\begin{lemma}\label{lem-PMP-reduced}
Let $\hat y_0\in L^2(\Omega;\mathbb R^k)\setminus\{0\}$. Then for
 each $T\in(0,\widehat T^*_{\hat y_0})$, there is a nontrivial multiplier $\eta_T \in L^2(\Omega;\mathbb R^k)\setminus\{0\}$, with the property
\begin{eqnarray}\label{adjoint-eq-reduced}
 f_T(\cdot):=
 \chi_\omega B_1^{\top} e^{(T-\cdot) (\mathbb I_k\triangle+A_1^{\top} )} \eta_T\in L^1(0,T;L^2(\Omega;\mathbb R^m))\setminus\{0\},
\end{eqnarray}
so that if $\hat u_{\hat{y}_0}$ is an optimal control to $(\widehat{\mathcal{TP}})_{\hat{y}_0}$, then
\begin{eqnarray}\label{PMP-eq-reduced}
    \big\langle \hat u_{\hat{y}_0}(t),f_T(t) \big\rangle_{L^2(\Omega;\mathbb R^m)}
=\max_{v\in B^m_1(0)}  \big\langle v, f_T(t)
\big\rangle_{L^2(\Omega;\mathbb R^m)}
\;\;\mbox{for a.e.}\;\;
t\in (0,T).
\end{eqnarray}
\end{lemma}

\begin{proof}
Our proof is based on the method provided in \cite[Theorems 4.3 and 4.4]{Wang-Wang-Xu-Zhang}.
    Given $0<t_1<t_2<+\infty$, define  the following controllable subspace  over $(t_1,t_2)$ with constrained controls:
\begin{eqnarray*}\label{yu-5-23-2}
    \widehat{Y}_C(t_1,t_2)&:=& \biggl\{f\in L^2(\Omega;\mathbb R^k) ~:~
\exists\; u\in L^\infty(\mathbb{R}^+; B^m_1(0)) \;\;\mbox{s.t.}
\nonumber\\
&\;&e^{(t_2-t_1)(\mathbb{I}_k\triangle+A_1)}f
+\int_{t_1}^{t_2}\chi_{\omega}B_1e^{(t_2-t)(\mathbb{I}_k\triangle+A_1)}u(t)dt=0\biggl\}.
\end{eqnarray*}
  Given $t>0$,  define the following reachable set of the  system (\ref{yu-5-18-9}):
\begin{eqnarray*}\label{yu-5-23-3}
    \widehat{Y}_R(t;\hat{y}_0):=\{\hat{y}(t;\hat{y}_0,u) ~:~
     u\in L^\infty(\mathbb{R}^+;B^m_1(0))\}.
\end{eqnarray*}
    Since $\mbox{rank}\;(B_1,A_1B_1,\cdots,A_1^{k-1}B_1)=k$ (see (\ref{yu-5-18-3})),
    it follows  by
    Proposition \ref{yu-lemma-5-18-2} and \cite[Theorem 4.4]{Wang-Wang-Xu-Zhang} that for each
    $t\in(0,\widehat{T}^*_{\hat{y}_0})$, $\widehat{Y}_R(t;\hat{y}_0)$ and $\widehat{Y}_C(t,\widehat{T}^*_{\hat{y}_0})$ are separable in $(L^2(\Omega))^k$.
    Thus, by \cite[Theorem 4.3]{Wang-Wang-Xu-Zhang},
    for each $T>0$, there is $\eta_T\in L^2(\Omega;\mathbb{R}^k)\setminus\{0\}$, with (\ref{adjoint-eq-reduced}), so that  (\ref{PMP-eq-reduced}) holds for any optimal
     control to  the problem $(\widehat{\mathcal{TP}})_{\hat{y}_0}$ (if it has an optimal control).
     This completes the proof of Lemma \ref{lem-PMP-reduced}.
\end{proof}

We are now in the position to prove Proposition \ref{thm-TP-PMP}.
\begin{proof}[Proof of Proposition \ref{thm-TP-PMP}]
Arbitrarily fix $y_0\in L^2(\Omega;\mathbb R^n)\setminus\{0\}$. Let $\hat y_0\in L^2(\Omega;\mathbb R^k)\setminus\{0\}$ satisfy
    $P^{-1}y_0=(\hat y_0,0)^{\top}$.
Then according to Lemma \ref{yu-proposition-5-18-1}, the following conclusions  are true:
\begin{enumerate}
  \item[] (C1) ~~$T^*_{y_0} = \widehat T^*_{\hat y_0}$;
  \item[] (C2) ~~$(\mathcal{TP})_{y_0}$ and $(\widehat{\mathcal{TP}})_{\hat{y}_0}$ has same optimal controls,
  if one of them has optimal controls.
\end{enumerate}
Arbitrarily fix $T\in (0,T^*_{y_0})$.
Then by the above (C1), we have $T\in (0,\widehat T^*_{\hat y_0})$.
By this, we can apply Lemma \ref{lem-PMP-reduced} to find a
 a multiplier
 $\eta_T \in L^2(\Omega;\mathbb R^k)\setminus\{0\}$, with  (\ref{adjoint-eq-reduced}),
 so that (\ref{PMP-eq-reduced}) holds for any optimal control to $(\widehat{\mathcal{TP}})_{\hat{y}_0}$ (if it has optimal controls).

 Let
 \begin{eqnarray}\label{yu-7-2-1}
  \xi_T := (P^{-1})^\top(\eta_T,0)^\top(=P(\eta_T,0)^\top).
 \end{eqnarray}
 Then by  the second equality in (\ref{relation-C0-two}), we see  that for each $t\in(0,T)$,
 \begin{eqnarray}\label{2.28,7.5}
  \mathcal B^* e^{(T-t)\mathcal A^*}\xi_T &=&
  \mathcal B^* e^{(T-t)\mathcal A^*} (P^{-1})^\top (\eta_T,0)^\top
  = \left( P^{-1} e^{(T-t)\mathcal A} \mathcal B\right)^* (\eta_T,0)^\top
  \nonumber\\
  &=&\left( \chi_\omega B_1^\top e^{(T-t)(\mathbb I_k\triangle+ A_1^\top)}, 0 \right)
  (\eta_T,0)^\top
  =\chi_\omega B_1^\top e^{(T-t)(\mathbb I_k\triangle+ A_1^\top)} \eta_T.
 \end{eqnarray}
 From (\ref{adjoint-eq-reduced}) and \eqref{2.28,7.5}, we see that $\xi_T$ satisfies
 (\ref{adjoint-eq}), while from (\ref{adjoint-eq-reduced}) and the above (C2), we find that
 (\ref{PMP-eq}), with the above $\xi_T$, holds for any optimal control to $(\mathcal{TP})_{y_0}$
 (if it has optimal controls).

 Finally,
  by the first equality in (\ref{yu-5-18-2}), (\ref{yu-7-2-1}) and \eqref{yu-5-16-5}, we have $\xi_T\in\mathfrak{L}\setminus\{0\}$. This ends the proof of Proposition \ref{thm-TP-PMP}.
\end{proof}

\section{Proof of main results}

We are now on the position to prove Theorem \ref{yu-theorem-5-24-1}.
\begin{proof}[Proof of Theorem \ref{yu-theorem-5-24-1}]
Arbitrarily fix $y_0\in L^2(\Omega;\mathbb R^n)\setminus\{0\}$ which satisfies $(\textbf{A})_{y_0}$.
Then from the note ($a_1$) in Section 1.1, the problem
$(\mathcal{TP})_{y_0}$ has at least one optimal control, from which, one can easily check that
$T^*_{y_0}>0$.

Before, proving the theorem, we will show a {\it key conclusion}. To state it,
we arbitrarily fix $T\in (0,T^*_{y_0})$. According to Proposition \ref{thm-TP-PMP}, there is $\xi_T\in\mathfrak{L}\setminus\{0\}$ so that (\ref{PMP-eq}) holds.
  Then by (\ref{adjoint-eq}) in Proposition \ref{thm-TP-PMP} and Corollary \ref{cor-switch-behavior}, we
  have
  $$
  \sharp[\{t\in(0,T):\mathcal{B}^*e^{(T-t)\mathcal{A}^*}\xi_T=0\}]<+\infty.
  $$
  So we can find $p:=p(\xi_T)\in\mathbb{N}$ so that
\begin{eqnarray}\label{3.1,7.5wang}
    \mathfrak{T}_{\xi_T}:=\{t\in(0,T):\mathcal{B}^*e^{(T-t)\mathcal{A}^*}\xi_T=0\}\cup\{t_0=0\}
    :=\{t_i\}_{i=0}^p.
   \end{eqnarray}

\par
    The above-mentioned {\it key conclusion} is as: there is a unique left-continuous function $f_{\xi_T}$ in $\mathcal{PC}([0,T);B^m_1(0))$ so that
\begin{eqnarray}\label{def-f-xiT}
f_{\xi_T}(t)=
 \frac{ \mathcal B^*e^{(T-t)\mathcal A^*} \xi_{T} }
 { \|\mathcal B^*e^{(T-t)\mathcal A^*} \xi_{T} \|_{L^2(\Omega;\mathbb R^m)} },
 ~~&t\in[0,T)\setminus \mathfrak{T}_{\xi_T}.
\end{eqnarray}
  To show \eqref{def-f-xiT}, it suffices to prove  that for each $t_i$ ($0\leq i\leq p$),
\begin{equation}\label{yu-5-24-8}
    \lim_{t\to t_i^-}\frac{\mathcal{B}^* e^{(T-t)\mathcal{A}^*}\xi_T}
{\|\mathcal{B}^*e^{(T-t)\mathcal{A}^*}\xi_T\|_{L^2(\Omega;\mathbb{R}^m)}}\;\;\mbox{exists}
\end{equation}
    and
\begin{equation}\label{yu-5-25-9}
    \lim_{t\to t_i^+}\frac{\mathcal{B}^* e^{(T-t)\mathcal{A}^*}\xi_T}
{\|\mathcal{B}^* e^{(T-t)\mathcal{A}^*}\xi_T\|_{L^2(\Omega;\mathbb{R}^m)}}\;\;\mbox{exists}.
\end{equation}
   To show    \eqref{yu-5-24-8} and \eqref{yu-5-25-9}, we arbitrarily fix $0\leq i\leq p$ and  divide the proof into several steps.
\vskip 5pt
   \noindent \emph{Step 1.  We prove that
\begin{equation}\label{yu-5-24-9}
    \chi_{\omega}e^{(T-t)\mathbb{I}_n\triangle}\xi_T\neq 0\;\;\mbox{for any}\;\;t\in[0,T).
\end{equation}}
\par
    Actually, if there exists a $\hat{t}\in[0,T)$ such that $\chi_{\omega}e^{\mathbb{I}_n\triangle (T-\hat{t})}\xi_T=0$, then by Proposition \ref{thm-sampling-unique} (with $B:=\mathbb{I}_n$ and $A:=0$), we can get that $\xi_T=0$, which  contradicts  our assumption. Hence (\ref{yu-5-24-9}) holds.
\vskip 5pt
    \noindent \emph{Step 2. We define two numbers.}

 By (\ref{yu-5-24-9}), we have that, for each $i\in\{0,1,\ldots,p\}$,
   \begin{equation}\label{3.5,7.5}
   \{j\in\mathbb{N}:\chi_{\omega}B^\top (A^\top)^je^{(T-t_i)\mathcal{A}^*}\xi_T\neq 0\}\neq \emptyset.
   \end{equation}
   and
   \begin{equation}\label{3.6,7.5}
   \{j\in\mathbb{N}:B^\top (A^\top)^je^{(T-t_i)\mathcal{A}^*}\xi_T\neq 0\}\neq \emptyset.
   \end{equation}
   Indeed, to show \eqref{3.5,7.5}, by the Hamilton-Cayley theorem, we only need to prove that, for each $\delta\in (0,T)$
\begin{equation}\label{yu-7-5-1}
    B^\top e^{(T-\cdot)A^\top}\chi_{\omega}e^{\mathbb{I}_n\triangle (T-t_i)}\xi_T\neq 0\;\;\mbox{in}\;\;[T-\delta,T].
\end{equation}
Since $\xi_T\in\mathfrak{L}\setminus\{0\}$ (see Proposition \ref{thm-TP-PMP}), there is $\eta_T\in L^2(\Omega;\mathbb{R}^k)$ so that
    $\xi_T=(P^{-1})^{\top}(\eta_T,0)^\top$. Thus, if there are $\delta\in(0,T)$ and $i\in\{0,1,\ldots,p\}$ so that (\ref{yu-7-5-1}) does not hold, then, by (\ref{yu-5-18-2}),
        for each $t\in[T-\delta,T]$,
\begin{eqnarray*}
    0&=&B^\top e^{(T-t)A^\top}\chi_{\omega}e^{(T-t_i)\mathbb{I}_n\triangle}\xi_T =B^\top e^{(T-t)A^\top}
    \chi_{\omega}e^{(T-t_i)\mathbb{I}_n\triangle}(P^{-1})^\top (\eta_T,0)^\top\nonumber\\
    &=&(P^{-1}e^{(T-t)A^\top}B)^\top (\chi_{\omega}e^{(T-t_i)\mathbb{I}_k\triangle}\eta_T,0)^\top
    =B_1^\top e^{(T-t)A^\top_1}\chi_{\omega} e^{(T-t_i)\mathbb{I}_k\triangle}\eta_T.
\end{eqnarray*}
    Because of $\mbox{rank}\,(B_1,A_1B_1,\cdots,A^{k-1}_1B_1)=k$, it follows that $\chi_{\omega} e^{(T-t_i)\mathbb{I}_k\triangle}\eta_T=0$, which implies that $\chi_{\omega}e^{(T-t_i)\mathbb{I}_n\triangle}\xi_T= 0$. It contradicts to (\ref{yu-5-24-9}).
    Thus, (\ref{yu-7-5-1}) is true. Moreover, it is obvious that, if (\ref{3.5,7.5}) holds, then (\ref{3.6,7.5}) is true.
\par

   Thus, we can define two integers in $\mathbb{N}^+$ by
\begin{equation}\label{yu-5-24-11}
    p(i):=\min\{j\in\mathbb{N}:\chi_{\omega}B^\top (A^\top)^je^{(T-t_i)\mathcal{A}^*}\xi_T\neq 0\}
\end{equation}
    and
\begin{equation}\label{yu-5-24-11-b}
    \hat{p}(i):=\min\{j\in\mathbb{N}:B^\top (A^\top)^je^{(T-t_i)\mathcal{A}^*}\xi_T\neq 0\}.
\end{equation}
    (By (\ref{yu-7-5-1}) and the Hamilton-Cayley theorem, we have $p(i)\leq n-1$ and $\hat{p}(i)\leq n-1$.)
\vskip 5pt
   \noindent \emph{Step 3. We prove that
\begin{equation}\label{yu-5-24-11-bb}
    p(i)=\hat{p}(i).
\end{equation}}
\par
    By \eqref{yu-5-24-11} and \eqref{yu-5-24-11-b},
    we have
     $\hat{p}(i)\leq p(i)$. We now show that $p(i)\leq \hat{p}(i)$. For this purpose, we only need to prove that
\begin{equation}\label{yu-5-24-11-bb-1}
    \chi_{\omega}B^\top (A^\top)^{\hat{p}(i)}e^{(T-t_i)\mathcal{A}^*}\xi_T\neq 0.
\end{equation}
    To show \eqref{yu-5-24-11-bb-1}, we first use \eqref{yu-5-24-11-b}  to see
\begin{equation}\label{yu-m-bb-1}
   \varphi_T:= B^\top (A^\top)^{\hat{p}(i)}e^{(T-t_i)A^\top}\xi_T\neq 0.
\end{equation}
    We next to observe that
    $$
    \varphi(\cdot):=B^\top (A^\top)^{\hat{p}(i)}e^{(T-t_i)A^\top}e^{(T-\cdot)\mathbb{I}_n\triangle}\xi_T
=e^{(T-\cdot)\mathbb{I}_m\triangle}\varphi_T,
$$
   from which, it follows that  $\varphi(\cdot)$ is the solution the following equation:
\begin{equation*}\label{yu-5-24-11-bb-2}
\begin{cases}
    \varphi_t=-\mathbb{I}_m\triangle\varphi&\mbox{in}\;\;\Omega\times(0,T),\\
    \varphi=0&\mbox{on}\;\;\partial\Omega\times (0,T),\\
    \varphi(T)=\varphi_T.
\end{cases}
\end{equation*}

Now, by contradiction, we suppose that (\ref{yu-5-24-11-bb-1}) is not true. Then
 we have $\varphi(t_i)=0$. This, together with Proposition \ref{thm-sampling-unique} (with $B:=\mathbb{I}_m$ and $A:=0$), gives that $\varphi_T=0$,
 which  contradicts  (\ref{yu-m-bb-1}). So, (\ref{yu-5-24-11-bb-1}) is proved and  (\ref{yu-5-24-11-bb}) is true.
\vskip 5pt
   \noindent \emph{Step 4. We give an expression on $B^\top e^{(T-t)A^\top}e^{(T-t_i)\mathbb{I}_n\triangle}\xi_T$. }
\par
  By (\ref{yu-5-24-11-bb}) and  (\ref{yu-5-24-11-b}), one can directly check that for any $t\in(-1,T)$,
\begin{eqnarray}\label{yu-5-24-12}
    &\;&B^\top e^{(T-t)A^\top}e^{(T-t_i)\mathbb{I}_n\triangle}\xi_T\nonumber\\
    &=&\sum_{j=0}^{+\infty}(t-t_i)^j\frac{B^\top (-A^\top)^je^{(T-t_i)A^\top}}{j!}e^{(T-t_i)\mathbb{I}_n\triangle}\xi_T\nonumber\\
    &=&(t-t_i)^{p(i)}\sum_{j=p(i)}^{+\infty} (t-t_i)^{j-p(i)}\frac{B^\top (-A^\top)^je^{(T-t_i)A^\top}}{j!}e^{(T-t_i)\mathbb{I}_n\triangle}\xi_T\nonumber\\
&=&(t-t_i)^{p(i)}[a_{p(i)}+(t-t_i)b_{p(i)}(t)],
\end{eqnarray}
    where
\begin{equation}\label{yu-5-24-13}
    a_{p(i)}:=\frac{B^\top(-A^\top)^{p(i)}e^{(T-t_i)A^\top}}{p(i)!}
e^{(T-t_i)\mathbb{I}_n\triangle}\xi_T,
\end{equation}
    and
\begin{equation}\label{yu-5-24-14}
    b_{p(i)}(t):=\sum_{j=p(i)+1}^{+\infty}
    (t-t_i)^{j-p(i)-1}\frac{B^\top (-A^\top)^je^{(T-t_i)A^\top}}{j!}e^{(T-t_i)\mathbb{I}_n\triangle}\xi_T.
\end{equation}
    From (\ref{yu-5-24-11}) and (\ref{yu-5-24-11-bb}), we have
\begin{equation}\label{yu-5-24-15}
    \chi_{\omega} a_{p(i)}\neq 0,
\end{equation}
   while from  \eqref{yu-5-24-14}, we can find $C>0$ so that
\begin{equation}\label{yu-5-24-16}
    \|b_{p(i)}(t)\|_{L^2(\Omega;\mathbb{R}^m)}\leq C \|\xi_T\|_{L^2(\Omega;\mathbb{R}^n)}\;\;\mbox{for all}\;\;t\in[0,T].
\end{equation}
\par
  \vskip 5pt
   \noindent \emph{Step 5. We prove that there is $\delta_0>0$ so that
\begin{equation}\label{yu-5-27-30}
    B^\top e^{(T-t)\mathcal{A}^*}\xi_T
    =B^\top e^{(T-t)A^\top}e^{(T-t_i)\mathbb{I}_n\triangle}\xi_T
    +O((t-t_i)^{p(i)+1})\;\;\mbox{for any}\;\;t\in (t_i-\delta_0, t_i+\delta),
\end{equation}
    where and in what follows,  $O(s^q)$, with $q\in\mathbb{N}^+$, stands for  a function
    $f:\mathbb{R}^+\to (L^2(\Omega))^m$ so that
    $\|f(s)\|_{(L^2(\Omega))^m}\leq Cs^q$ for some constant $C>0$.}

   Given  $\varepsilon\in (0,T)$ so that $t_i\in (-1,T-\varepsilon)$, we take
   $\delta>0$ so that $(t_i-\delta,t_i+\delta)\subset (-1,T-\varepsilon)$.
    Since the operator $\mathbb{I}_n\triangle$, with its domain $D(\mathbb{I}_n\triangle)=H_0^1(\Omega;\mathbb{R}^n)\cap H^2(\Omega;\mathbb{R}^n)$,
         generates an analytical semigroup $\{e^{t\mathbb{I}_n\triangle}\}_{t\geq 0}$,
         we can use the properties of analytic semigroups (see \cite[Chapter 2, Section 2.5]{Pazy})
         to find  $\widetilde{C}>0$ so that
\begin{eqnarray}\label{yu-5-24-19}
    \|(\mathbb{I}_n\triangle)^je^{(T-t_i)\mathbb{I}_n\triangle }\|_{\mathcal{L}(L^2(\Omega;\mathbb{R}^n);L^2(\Omega;\mathbb{R}^n)}
    \leq j!\left(\frac{\widetilde{C}e}{T-t_i}\right)^j\leq  j!(\widetilde{C}e\varepsilon^{-1})^j
    \;\mbox{for all}\; j\in \mathbb{N}.
\end{eqnarray}
     (In \eqref{yu-5-24-19}, we used the fact $T-t_i>\varepsilon$.) Thus, we have that
     when $t\in (t_i-\hat\delta, t_i+\hat\delta)$, with $\hat{\delta}:=\min\{\frac{\varepsilon}{2\widetilde{C}e},\delta\}$,
\begin{equation}\label{yu-5-27-20}
    e^{(T-t)\mathbb{I}_n\triangle}\xi_T=e^{(T-t_i)\mathbb{I}_n\triangle}\xi_T
    +\sum_{j=1}^{+\infty}(t-t_i)^j\frac{(-\mathbb{I}_n\triangle)^j}{j!}
e^{(T-t_i)\mathbb{I}_n\triangle}\xi_T.
\end{equation}
     (Notice that \eqref{yu-5-24-19} ensures the convergence of the series in \eqref{yu-5-27-20} in
     $L^2(\Omega;\mathbb{R}^n)$.) Now,  by (\ref{yu-5-27-20}) and (\ref{yu-5-24-12}), we see that for each
     $t\in (t_i-\hat\delta,t_i+\hat\delta)$,
     \begin{eqnarray}\label{yu-5-27-21}
    &\;&B^\top e^{A^\top (T-t)}(e^{(T-t)\mathbb{I}_n\triangle }-e^{(T-t_i)\mathbb{I}_n\triangle})\xi_T\nonumber\\
    &=&(t-t_i)^j\sum_{j=1}^{+\infty}\frac{(-\mathbb{I}_n\triangle)^j}{j!}
B^\top e^{(T-t)A^\top}e^{(T-t_i)\mathbb{I}_n\triangle}\xi_T\nonumber\\
    &=& \sum_{j=1}^{+\infty}(t-t_i)^{p(i)+j}\frac{(-\mathbb{I}_n\triangle)^j}{j!}
[a_{p(i)}+(t-t_i)b_{p(i)}(t)].
\end{eqnarray}
   While by (\ref{yu-5-24-13}), (\ref{yu-5-24-14}) and (\ref{yu-5-24-19}), one can easily check that there is a constant $C_{p(i)}>0$ so that
\begin{equation*}\label{yu-5-27-22}
    \|(-\mathbb{I}_n\triangle)^ja_{p(i)}\|_{L^2(\Omega;\mathbb{R}^m)}\leq C_{p(i)}j!(\widetilde{C}e\varepsilon^{-1})^j\|\xi_T\|_{L^2(\Omega;\mathbb{R}^n)}\;\;\mbox{for all}\;\;j\in \mathbb{N}
\end{equation*}
    and that when $t\in (t_i-\tilde{\delta}, t_i+\tilde{\delta})$ with $\tilde{\delta}:=\min\{\hat{\delta},\frac{\varepsilon}{2\|A\|_{\mathcal{L}(\mathbb{R}^n;\mathbb{R}^n)}
\widetilde{C}e}\}$,
\begin{eqnarray*}\label{yu-5-27-23}
    &\;&\|(-\mathbb{I}_n\triangle)^jb_{p(i)}(t)\|_{(L^2(\Omega))^m}\nonumber\\
    &\leq&
    \|B\|_{\mathcal{L}(\mathbb{R}^m;\mathbb{R}^n)}
\|A\|^{p(i)}_{\mathcal{L}(\mathbb{R}^n;\mathbb{R}^n)}
    \|\xi_T\|_{L^2(\Omega;\mathbb{R}^n)}\sum_{p=1}^{+\infty}
    (t-t_i)^p(\|A\|_{\mathcal{L}(\mathbb{R}^n;\mathbb{R}^n)}\widetilde{C}e\varepsilon^{-1})^p\nonumber\\
    &\leq&\|B\|_{\mathcal{L}(\mathbb{R}^m;\mathbb{R}^n)}
\|A\|^{p(i)}_{\mathcal{L}(\mathbb{R}^n;\mathbb{R}^n)}
    \|\xi_T\|_{L^2(\Omega;\mathbb{R}^n)}\sum_{p=1}^{+\infty}\left(\frac{1}{2}\right)^p<+\infty.
\end{eqnarray*}
    These imply that when $t\in (t_i-\tilde{\delta}, t_i+\tilde{\delta})$,
\begin{equation*}\label{yu-5-27-24}
    \sum_{j=1}^{+\infty}(t-t_i)^{p(i)+j}\frac{(-\mathbb{I}_n\triangle)^j}{j!}
[a_{p(i)}+(t-t_i)b_{p(i)}(t)]=O((t-t_i)^{p(i)+1}).
\end{equation*}
      This, together with (\ref{yu-5-27-21}), yields (\ref{yu-5-27-30}), with $\delta_0:=\tilde{\delta}$.

   \noindent \emph{Step 6. We prove \eqref{yu-5-24-8}.}

\par
      By (\ref{yu-5-24-12}) and (\ref{yu-5-27-30}), we see that when  $t\in (t_i-\delta_0, t_i+\delta_0)$,
\begin{eqnarray}\label{yu-5-25-3}
    &\;&\mathcal{B}^* e^{(T-t)\mathcal{A}^*}\xi_T\nonumber\\
&=& \mathcal{B}^* e^{(T-t)A^\top}e^{(T-t_i)\mathbb{I}_n\triangle}\xi_T
+\mathcal{B}^* e^{(T-t)A^\top}(e^{(T-t)\mathbb{I}_n\triangle}-e^{(T-t_i)\mathbb{I}_n\triangle})\xi_T\nonumber\\
    &=&
    (t-t_i)^{p(i)}\chi_{\omega}[a_{p(i)}+(t-t_i)b_{p(i)}(t)]
    +\chi_{\omega}O((t-t_i)^{p(i)+1})\nonumber\\
    &=&(t-t_i)^{p(i)}\chi_{\omega}[a_{p(i)}+(t-t_i)b_{p(i)}(t)
+O((t-t_i))].
\end{eqnarray}
    Hence, for each $t\in(-1,t_i)\cap O_{\delta_0}(t_i)$,
\begin{eqnarray}\label{yu-5-25-4}
    &\;&\frac{\mathcal{B}^* e^{(T-t)\mathcal{A}^*}\xi_T}
{\|\mathcal{B}^* e^{(T-t)\mathcal{A}^*}\xi_T\|_{L^2(\Omega;\mathbb{R}^m)}}\nonumber\\
&=&
    \frac{(t-t_i)^{p(i)}\chi_{\omega}[a_{p(i)}+(t-t_i)b_{p(i)}(t)+O((t_i-t))]}
{|t-t_i|^{p(i)}\|\chi_{\omega}[a_{p(i)}+(t-t_i)b_{p(i)}(t))+O((t_i-t))]\|
_{L^2(\Omega;\mathbb{R}^m)}}\nonumber\\
&=&
\begin{cases}
    \displaystyle\frac{\chi_{\omega}[a_{p(i)}+(t-t_i)b_{p(i)}(t)+O((t_i-t))]}
{\|\chi_{\omega}[a_{p(i)}+(t-t_i)b_{p(i)}(t)+O((t_i-t))]\|
_{L^2(\Omega;\mathbb{R}^m)}},&\mbox{if}\;\;p(i)\;\;\mbox{is even},\\
-\displaystyle\frac{\chi_{\omega}[a_{p(i)}+(t-t_i)b_{p(i)}(t)+O((t_i-t))]}
{\|\chi_{\omega}[a_{p(i)}+(t-t_i)b_{p(i)}(t)+O((t_i-t))]\|
_{L^2(\Omega;\mathbb{R}^m)}},&\mbox{if}\;\;p(i)\;\;\mbox{is odd}.
\end{cases}
\end{eqnarray}
     By sending $t\to t_i^-$ in (\ref{yu-5-25-4}) and using (\ref{yu-5-24-15}) and (\ref{yu-5-24-16}), we obtain that
\begin{equation}\label{yu-5-25-5}
    \lim_{t\to t_i^-}\frac{\mathcal{B}^* e^{(T-t)\mathcal{A}^*}\xi_T}
{\|\mathcal{B}^*e^{(T-t)\mathcal{A}^*}\xi_T\|_{L^2(\Omega;\mathbb{R}^m)}}
    =
\begin{cases}
    \displaystyle\frac{\chi_{\omega}a_{p(i)}}{\|\chi_{\omega}a_{p(i)}\|_{L^2(\Omega;\mathbb{R}^m)}},
&\mbox{if}\;\;p(i)\;\;\mbox{is even},\\
    -\displaystyle\frac{\chi_{\omega}a_{p(i)}}{\|\chi_{\omega}a_{p(i)}\|_{L^2(\Omega;\mathbb{R}^m)}},
&\mbox{if}\;\;p(i)\;\;\mbox{is odd}.
\end{cases}
\end{equation}
    This leads to (\ref{yu-5-24-8}).
\vskip 5pt
   \noindent \emph{Step 7. The proof of (\ref{yu-5-25-9}).}
\par
    By (\ref{yu-5-25-3}), (\ref{yu-5-24-15}) and (\ref{yu-5-24-16}), we see that
\begin{equation}\label{yu-5-25-14}
    \lim_{t\to t_i^+}
    \frac{\mathcal{B}^*e^{(T-t)\mathcal{A}^*}\xi_T}
{\|\mathcal{B}^* e^{(T-t)\mathcal{A}^*}\xi_T\|_{L^2(\Omega;\mathbb{R}^m)}}=\frac{\chi_{\omega}a_{p(i)}}
{\|\chi_{\omega}a_{p(i)}\|_{L^2(\Omega;\mathbb{R}^m)}},
\end{equation}
   which leads to (\ref{yu-5-25-9}).
\vskip 5pt
In summary, we conclude that the {\it key conclusion} has been proved.

\vskip 5pt

{\it We now show  that  $(\mathcal{TP})_{y_0}$ has a unique optimal control, which has the bang-bang property.}
To this end, we let  $\hat u$ be an optimal control to $(\mathcal{TP})_{y_0}$.
Given $T\in(0,T^*_{y_0})$, let $\xi_T$ be given by Proposition \ref{thm-TP-PMP}. Then by (\ref{PMP-eq}) and the
{\it key conclusion}, we have
\begin{eqnarray}\label{optimal-PMP-representation}
 \hat u|_{(0,T)}(t) =   f_{\xi_T}(t)
 ~~\mbox{for a.e.}\;\; t\in (0,T),
\end{eqnarray}
where $f_{\xi_T}$ is given by (\ref{def-f-xiT}). Since $f_{\xi_T}$ is in $\mathcal{PC}([0,T);B_1^m(0))$
(This follows from (\ref{def-f-xiT}) and \eqref{3.1,7.5wang}.)
and because \eqref{optimal-PMP-representation} holds for each $T\in(0,T^*_{y_0})$, (Notice that when $T$ varies, $f_{\xi_T}$, as well as $\xi_T$, changes.), we have
\begin{eqnarray}\label{f-uniqueness}
 f_{\xi_T}|_{(0,S)}=f_{\xi_S}
 \;\;\mbox{when}\;\;
 0<S<T<T^*_{y_0},
\end{eqnarray}
    where $\xi_S$ is given in Proposition \ref{thm-TP-PMP}, where $T$ is replaced by  $S$.

    Meanwhile, it follows  from (\ref{adjoint-eq}) in Proposition \ref{thm-TP-PMP} and Corollary \ref{cor-switch-behavior} that
\begin{eqnarray}\label{I-xiT-uniform-upper-bound}
\sup_{ 0<S<T^*_{y_0} } \sharp\, \left[\bigg\{ t\in (0,S)~:~ \mathcal B^*e^{(S-t) \mathcal A^*}\xi_S  = 0 \bigg\}\right]
<+\infty.
\end{eqnarray}
  (Here, we note that $T^*_{y_0}<+\infty$.) Define a control $u^*_{y_0}$ in the following manner: For each $t\in (0,T^*_{y_0})$, we arbitrarily
take $T\in (t,  T^*_{y_0})$ and then define
\begin{eqnarray}\label{def-unique-optimal-control}
 u^*_{y_0}(t):= f_{\xi_T}(t).
\end{eqnarray}
 By (\ref{f-uniqueness}), we see that $u^*_{y_0}$ is well defined. From \eqref{def-unique-optimal-control} and (\ref{optimal-PMP-representation}), it follows that
\begin{eqnarray*}
 u^*_{y_0}(t)=\hat u(t)
 ~~\mbox{for a.e.}\;\; t\in (0,T^*_{y_0}).
\end{eqnarray*}
Thus, $(\mathcal{TP})_{y_0}$ has a unique optimal control $u^*_{y_0}$. Moreover, from (\ref{def-unique-optimal-control}) and (\ref{def-f-xiT}), $u^*_{y_0}$ has the bang-bang property.
\par
    Finally, since  $f_{\xi_T}\in \mathcal{PC}([0,T);B^m_1(0))$ for any $T\in(0,T^*_{y_0})$,
    it follows from (\ref{def-unique-optimal-control}), (\ref{I-xiT-uniform-upper-bound}) that the above control $u^*_{y_0}$ is in  the space $\mathcal{PC}([0,T^*_{y_0});B^m_1(0))$.
    Thus we finish the proof of the conclusion $(i)$.

\vskip 5pt
    \emph{We next show $(ii)$.} Let $I$ be an open subinterval of $(0,T^*_{y_0})$ with $|I|\leq d_A$. We aim to show that the optimal control $u^*_{y_0}$ defined in (\ref{def-unique-optimal-control}) has at most $q_{A,B}-1$ switching points over $I$. By contradiction, we suppose that it was not true. Then there would be a set $\{t_j\}_{j=1}^{q_{A,B}}\subset I$ so that each $t_j$  is a switching point of $u^*_{y_0}$.
     Let $\widehat T \in (0,T^*_{y_0})$ so that
 \begin{eqnarray*}
  \widehat T > \max_{1\leq j\leq q_{A,B}} t_j.
 \end{eqnarray*}
    Then from (\ref{def-unique-optimal-control}) and (\ref{def-f-xiT}), one has that
 \begin{eqnarray}\label{3,30,7.6}
u^*_{y_0}|_{(0,\widehat{T})}(t)=
 \frac{ \mathcal B^*e^{(\widehat T-t)\mathcal A^*} \xi_{\widehat T} }
 { \|\mathcal B^*e^{(\widehat T-t)\mathcal A^*} \xi_{\widehat T} \|_{L^2(\Omega;\mathbb R^m)} },
 ~~\mbox{when}\;t\in(0,\widehat T)\setminus \mathfrak{T}_{ \xi_{\widehat T} }.
\end{eqnarray}
     Notice that the function on the right hand side of \eqref{3,30,7.6} is continuous at each
     $t\in (0,\widehat T)\setminus \mathfrak{T}_{ \xi_{\widehat T} }$;
    $\{t_j\}_{j=1}^{q_{A,B}}\subset \mathfrak{T}_{ \xi_{\widehat T} }$; and each $t_j$, with
    $j=1,\dots, q_{A,B}$ is in $I$. Thus, we see from \eqref{3,30,7.6} that
\begin{eqnarray*}
 \mathcal B^*e^{(\widehat T-t)\mathcal A^*} \xi_{\widehat T}=0
 \;\;\mbox{for each}\;\;
 t\in \{t_j\}_{j=1}^{q_{A,B}}\subset I.
\end{eqnarray*}
This, along with Proposition \ref{thm-sampling-unique}, yields
\begin{eqnarray*}
  \mathcal B^*e^{(\widehat T-t)\mathcal A^*} \xi_{\widehat T}=0
  \;\;\mbox{for all}\;\;
 t\in (0,\widehat T),
\end{eqnarray*}
which contradicts (\ref{I-xiT-uniform-upper-bound}). Thus, the conclusion $(ii)$ is true.
\vskip 5pt
    \emph{We finally prove $(iii)$.}
     Indeed, by (\ref{def-unique-optimal-control}),
     we have that for each $\hat{t}\in(0,T^*_{y_0})$, if it is a switching point of the optimal control $u^*_{y_0}$, then there exist $T>\hat{t}$ and $\xi_T\in \mathfrak{L}\setminus\{0\}$ so that $\hat{t}\in \mathfrak{T}_{\xi_T}$, i.e., there is a $i^*\in\{1,2,\ldots,p\}$ so that $\hat{t}=t_{i^*}$, where $\mathfrak{T}_{\xi_T}$ is defined by (\ref{3.1,7.5wang}). By (\ref{yu-5-25-5}) and (\ref{yu-5-25-14}), we can conclude that
     $p(i^*)$, which is defined by (\ref{yu-5-24-11}), is an odd number and
     $\lim_{s\to t_{i^*}^+}u^*(s)+\lim_{s\to t_{i^*}^-}u^*(s)=0$. Thus $(iii)$ is true.

\par
    Hence, we complete the proof of Theorem \ref{yu-theorem-5-24-1}.
\end{proof}

\section{An example}
    In this section, we present an example to show that, in some time optimal control problem of coupled heat system, the switching phenomenon happens. From this perspective, the system \eqref{yu-5-16-1} differs from
    the pure heat equation, since the optimal control for the later has no any switching point (see the note $(b_3)$).
\vskip 5pt
    \textbf{Example.}
  Let $\omega=\Omega$. Let $n:=2$ and $m:=1$. Let
  \begin{equation}\label{yu-6-26-1}
    A:=\left(
         \begin{array}{cc}
           0 & 1 \\
           -1 & 0 \\
         \end{array}
       \right),\;\;B:=\left(
                        \begin{array}{c}
                          1 \\
                          0 \\
                        \end{array}
                      \right).
\end{equation}
    From \eqref{yu-6-26-1}, we can directly check that  $\mbox{rank}\,(B,AB)=2$ and $\sigma(A)=\{i,-i\}$,
    and that
    \begin{equation}\label{yu-6-26-2}
    e^{tA}=\left(
             \begin{array}{cc}
               \cos t & \sin t\\
               -\sin t & \cos t \\
             \end{array}
           \right)\;\;\mbox{for each}\;\;t\in\mathbb{R}.
\end{equation}
   Write $\lambda_i$ for the $i$-th eigenvalue of the operator $-\triangle$, with its domain $H_0^1(\Omega)\cap H^2(\Omega))$, and let $e_i$ be the corresponding normalized eigenvector.

    First,  according to  Proposition \ref{yu-lemma-5-18-2},  the system  (\ref{yu-5-16-1}),
     with the above $(A,B)$, is $L^\infty_{x,t}$ null controllable and $\mathfrak{L}=L^2(\Omega;\mathbb{R}^2)$. (Here, $\mathfrak{L}$ is defined by \eqref{yu-5-16-5}.)

     Second,  from  \eqref{yu-6-26-2},   one can directly check that
\begin{equation}\label{yu-6-26-3}
    \|e^{t\mathcal{A}}\|_{\mathcal{L}(L^2(\Omega;\mathbb{R}^2);L^2(\Omega;\mathbb{R}^2))}
    =\|e^{t(\mathbb{I}_2\triangle+A)}\|_{\mathcal{L}(L^2(\Omega;\mathbb{R}^2);L^2(\Omega;\mathbb{R}^2))}
    \leq e^{-\lambda_1t}\;\;\mbox{for all}\;\;t\in\mathbb{R}^+.
\end{equation}

    From these and  by the note $(a_1)$ in Section \ref{yu-subsection-1-1},
    we can see that the problem $(\mathcal{TP})_{y_0}$ has an admissible control for each $y_0\in L^2(\Omega;\mathbb{R}^2)$. Then by $(i)$ in Theorem \ref{yu-theorem-5-24-1}, $(\mathcal{TP})_{y_0}$ has a unique optimal control $u^*_{y_0}$ whose restriction over $[0,T^*_{y_0})$ is in $\mathcal{PC}([0,T^*_{y_0});B_1^1(0))$.
\par
    Let $\eta:=(\eta_1,\eta_2)^\top \in\mathbb{R}^2$ so that
\begin{equation}\label{yu-6-26-4}
    \|\eta\|_{\mathbb{R}^2}=\eta_1^2+\eta_2^2>\lambda_1^{-1}(e^{4\pi \lambda_1}-1).
\end{equation}
     Choose
    $y_0:=\eta e_1$. By the optimality of $(u^*_{y_0},T^*_{y_0})$, we have
\begin{eqnarray}\label{yu-6-b-26-2}
    0&=&e^{T^*_{y_0}\mathcal{A}}y_0
    +\int_0^{T^*_{y_0}}e^{(T^*_{y_0}-t)\mathcal{A}}\mathcal{B}u_{y_0}^*(t)dt\nonumber\\
    &=& e^{T^*_{y_0}(A-\lambda_1\mathbb{I}_2)}\eta e_1
    +\int_0^{T^*_{y_0}}e^{(T^*_{y_0}-t)A}Be^{(T^*_{y_0}-t)\mathbb{I}_2\triangle}u^*_{y_0}(t)dt.
\end{eqnarray}
    This, along with (\ref{yu-6-26-2}) and (\ref{yu-6-26-3}), gives that
\begin{equation*}
    e^{-\lambda_1 T^*_{y_0}}\|\eta\|_{\mathbb{R}^2}
    \leq \int_0^{T^*_{y_0}}e^{-\lambda_1(T^*_{y_0}-t)}dt=\lambda_1^{-1}(1-e^{-\lambda_1 T^*_{y_0}}).
\end{equation*}
    (Here, we used the facts that  $\|B\|_{\mathcal{L}(\mathbb{R};\mathbb{R}^2)}\leq 1$ and $\|e^{At}\|_{\mathcal{L}(\mathbb{R}^2;\mathbb{R}^2)}=1$ for each $t\in\mathbb{R}$.) The above leads to
$$
    \|\eta\|_{\mathbb{R}^2}\leq \lambda_1^{-1}(e^{\lambda_1T^*_{y_0}}-1).
$$
    This, together with (\ref{yu-6-26-4}), gives that
\begin{equation}\label{yu-6-26-5}
    T_{y_0}^*\geq \lambda_1^{-1}\ln(\lambda_1\|\eta\|_{\mathbb{R}^2}+1)>4\pi.
\end{equation}
    Let $\widehat{T}:=4\pi$. By Proposition \ref{thm-TP-PMP} and (\ref{yu-6-26-5}), there is a $\xi_{\widehat{T}}\in L^2(\Omega;\mathbb{R}^2)\setminus\{0\}$ so that (\ref{adjoint-eq}) and (\ref{PMP-eq}) hold. Thus, by (\ref{PMP-eq}), we have
\begin{equation}\label{yu-6-26-6}
    u^*_{y_0}|_{(0,\widehat{T})}(t)=
    \frac{\mathcal{B}^* e^{(\widehat{T}-t)\mathcal{A}^*}\xi_{\widehat{T}}}{\|\mathcal{B}^* e^{(\widehat{T}-t)\mathcal{A}^*}\xi_{\widehat{T}}\|_{L^2(\Omega)}}\;\;\mbox{for each}\;\;t\in(0,\widehat{T})\setminus\mathfrak{S}_{\widehat{T}},
\end{equation}
    where
\begin{equation}\label{yu-6-26-7}
    \mathfrak{S}_{\widehat{T}}:=\left\{t\in(0,\widehat{T}):
    \mathcal{B}^*e^{(\widehat{T}-t)\mathcal{A}^*}\xi_{\widehat{T}}=0\right\}.
\end{equation}
\par

{\it Our aim is to claim that $(i)$ the set $\mathfrak{S}_{\widehat{T}}$ is not empty;
$(ii)$ each $\hat t$ in this set is a switching point of the optimal control $u^*_{y_0}$. }

For this purpose, we first show
\begin{equation}\label{yu-6-26-b-1}
    u^*_{y_0}(t)=f(t)e_1\;\;\mbox{for each}\;\;t\in(0,T^*_{y_0}),
\end{equation}
    where  $\|f(t)\|_{\mathbb{R}}\leq 1$ a.e. $t\in(0,T^*_{y_0})$. Actually, if $u^*(t)=\sum_{i=1}^{+\infty}f_i(t)e_i$ with $\sum_{i=1}^{+\infty}\|f_i(t)\|_{\mathbb{R}}^2\leq 1$ a.e. $t\in(0,T^*_{y_0})$, then by (\ref{yu-6-b-26-2}), we have
\begin{equation*}\label{yu-6-26-b-3}
    0=e^{T^*_{y_0}(A-\lambda_1\mathbb{I}_2)}\eta e_1+
    \sum_{i=1}^{+\infty}\int_0^{T^*_{y_0}}e^{(T^*_{y_0}-t)(A-\lambda_i\mathbb{I}_2)}Bf_i(t)dte_i.
\end{equation*}
    Thus we have
\begin{equation*}\label{yu-6-26-b-4}
\begin{cases}
    e^{T^*_{y_0}(A-\lambda_1\mathbb{I}_2)}\eta e_1+
    \displaystyle\int_0^{T^*_{y_0}}e^{(T^*_{y_0}-t)(A-\lambda_1\mathbb{I}_2)}Bf_1(t)dte_1=0,\\
    \displaystyle\sum_{i=2}^{+\infty}\displaystyle\int_0^{T^*_{y_0}}e^{(T^*_{y_0}-t)(A-\lambda_i\mathbb{I}_2)}Bf_i(t)dte_i=0.
\end{cases}
\end{equation*}
    So the control $\hat{u}^*_{y_0}:=f_1e_1$ is also an optimal control to $(\mathcal{TP})_{y_0}$.
   By the uniqueness of the optimal control to $(\mathcal{TP})_{y_0}$ (see $(i)$ in Theorem \ref{yu-theorem-5-24-1}), we find that
\begin{equation*}\label{yu-6-26-b-5}
    u^*_{y_0}=\hat{u}^*_{y_0}.
\end{equation*}
    Hence (\ref{yu-6-26-b-1}) holds.
\par
    Next, we prove that
\begin{equation}\label{yu-6-26-b-6}
    \xi_{\widehat{T}}=\zeta e_1\;\;\mbox{for some}\;\;\zeta\in\mathbb{R}^2\setminus \{0\}.
\end{equation}
     For this purpose, we suppose that $\xi_{\widehat{T}}=\sum_{i=1}^{+\infty}\zeta_ie_i$ with $\{\zeta_i\}_{i\in\mathbb{N}^+}\subset\mathbb{R}^2$ and $\sum_{i=1}^{+\infty}\|\zeta_i\|_{\mathbb{R}^2}^2>0$.
    Thus, by (\ref{yu-6-26-6}), we obtain that
\begin{equation*}\label{yu-6-26-b-7}
    u_{y_0}^*|_{(0,\widehat{T})\setminus\mathfrak{S}_{\widehat{T}}}(t)=\|\mathcal{B}^* e^{(\widehat{T}-t)\mathcal{A}^*}\xi_{\widehat{T}}\|_{L^2(\Omega)}^{-1}\sum_{i=1}^{+\infty}B^\top e^{(\widehat{T}-t)(A^\top-\lambda_i\mathbb{I}_2)}\zeta_ie_i.
\end{equation*}
    This, along with (\ref{yu-6-26-b-1}), yields that
\begin{equation}\label{yu-6-26-b-8}
    \|\mathcal{B}^* e^{(\widehat{T}-t)\mathcal{A}^*}\xi_{\widehat{T}}\|_{L^2(\Omega)}^{-1}\sum_{i=2}^{+\infty}B^\top e^{(\widehat{T}-t)(A^\top-\lambda_i\mathbb{I}_2)}\zeta_ie_i=0\;\;\mbox{a.e.}\;\;t\in(0,\widehat{T}).
\end{equation}
    By Corollary \ref{cor-switch-behavior}, the set $\mathfrak{S}_{\widehat{T}}$ has at most finite elements.
    It follows from (\ref{yu-6-26-b-8}) that
\begin{equation}\label{yu-6-26-bb-1}
    B^\top e^{(\widehat{T}-t)(A^\top-\lambda_i\mathbb{I}_2)}\zeta_i=0\;\;\mbox{a.e.}\;\;t\in(0,\widehat{T})
    \;\;\mbox{for each}\;\;i\geq 2.
\end{equation}
    Since $\mbox{rank}\;(B,AB)=2$ (which implies that $\mbox{rank}\;(B,(A-\lambda_i\mathbb{I}_2)B)=2$ for each $i\in\mathbb{N}^+$), by (\ref{yu-6-26-bb-1}), we can conclude that
\begin{equation*}\label{yu-6-26-bb-2}
    \zeta_i=0\;\;\mbox{for each}\;\;i\geq 2.
\end{equation*}
    Thus, (\ref{yu-6-26-b-6}) holds.

    By (\ref{yu-6-26-b-6}), we can write
    $$
    \xi_{\widehat{T}}=(\zeta_1,\zeta_2)^\top e_1.
    $$
    Thus, by (\ref{yu-6-26-1}), after some simple computations, we get
\begin{equation}\label{yu-6-26-9}
    \mathcal{B}^*e^{(\widehat{T}-t)\mathcal{A}^*}\xi_{\widehat{T}}
    =e^{-\lambda_1(\widehat{T}-t)}e_1\left(\zeta_1\cos (\widehat{T}-t)-\zeta_2\sin(\widehat{T}-t)\right)
    =e^{-\lambda_1(\widehat{T}-t)}e_1\sqrt{|\zeta_1|^2+|\zeta_2|^2}\sin(\widehat{T}-t+\theta),
\end{equation}
    where $\theta:=\arctan(\zeta_1/\zeta_2)$. (Here, we permit $\theta=\frac{\pi}{2}$ if $\zeta_2=0$.)
   By (\ref{yu-6-26-6}) and (\ref{yu-6-26-9}), we have
\begin{equation}\label{yu-6-26-10}
    u^*_{y_0}|_{(0,\widehat{T})}(t)=
    \frac{e_1\sqrt{|\zeta_1|^2+|\zeta_2|^2}\sin(\widehat{T}-t+\theta)}
    {\sqrt{|\zeta_1|^2+|\zeta_2|^2}|\sin(\widehat{T}-t+\theta)|}
    \;\;\mbox{for each}\;\;\;\;t\in(0,\widehat{T})\setminus\mathfrak{S}_{\widehat{T}}.
\end{equation}

   Finally, by (\ref{yu-6-26-5}), (\ref{yu-6-26-7}) and (\ref{yu-6-26-9}), one can check easily that
\begin{equation*}\label{yu-6-26-11}
    \mathfrak{S}_{\widehat{T}}=\{t\in(0,\widehat{T}):\sin(\widehat{T}-t+\theta)=0\}\neq \emptyset,
\end{equation*}
 which leads to the claim $(i)$.
 While  by (\ref{yu-6-26-10}), we see that for each $\hat{t}\in\mathfrak{S}_{\widehat{T}}$,
\begin{equation*}\label{yu-6-26-12}
    \lim_{t\to \hat{t}^+}u^*_{y_0}(t)\neq \lim_{t\to \hat{t}^-}u^*_{y_0}(t),
\end{equation*}
   which leads to the claim $(ii)$.

\end{document}